\documentclass[a4paper,10pt,reqno]{amsart}

\usepackage[english]{babel}
\usepackage[utf8]{inputenc}

\usepackage{amssymb}
\usepackage{amsmath}
\usepackage{amsthm}
\usepackage{bbm}

\usepackage{enumerate}
\usepackage{tikz}
\usepackage{marginnote}
\usepackage{color}

\usepackage{csquotes}  
\newcommand\q{\enquote}

\usepackage{hyperref}   
\hypersetup{
    colorlinks=true,
    linkcolor=blue,
    filecolor=magenta,      
    urlcolor=blue,
		citecolor=red,
    pdftitle={Sharelatex Example},
}

\newcommand{\jgc}[1]{{\color{blue} #1}}            

\newcommand \K   {\mathcal{K}}
\newcommand \Kinf{\mathcal{K_\infty}}

\newcommand \qiq   {\quad\Iff\quad}

\newcommand \Iff   {\Leftrightarrow}

\newcommand{\C}{\mathbb{C}}
\newcommand{\N}{\mathbb{N}}

\newcommand{\R}{\mathbb{R}}

\newcommand{\Z}{\mathbb{Z}}

\newcommand{\calK}{\mathcal{K}}
\newcommand{\calL}{\mathcal{L}}

\newcommand{\calT}{\mathcal{T}}

\DeclareMathOperator{\id}{id} 
\DeclareMathOperator{\dist}{dist} 
\newcommand{\argument}{\mathord{\,\cdot\,}} 
\newcommand{\norm}[1]{\left\lVert #1 \right\rVert} 
\newcommand{\modulus}[1]{\left\lvert #1 \right\rvert} 

\newcommand{\intt}{{\rm int}\,}

\newcommand{\spr}{r} 
\newcommand{\essSpr}{r_{\operatorname{ess}}} 

\newcommand{\Implies}[2]{``(\ref{#1}) $\Rightarrow$ (\ref{#2})''}

\newcommand{\lel}{\left\langle}
\newcommand{\rir}{\right\rangle}
\newcommand{\scalp}[2]{ \lel #1, #2 \rir }

\theoremstyle{definition}
\newtheorem{definition}{Definition}[section]
\newtheorem{remark}[definition]{Remark}
\newtheorem{remarks}[definition]{Remarks}
\newtheorem{example}[definition]{Example}
\newtheorem{examples}[definition]{Examples}

\theoremstyle{plain}
\newtheorem{proposition}[definition]{Proposition}
\newtheorem{lemma}[definition]{Lemma}
\newtheorem{theorem}[definition]{Theorem}

\numberwithin{equation}{section}

\begin{document}

\title[Stability of positive linear systems]{Stability criteria for positive linear discrete-time systems}

\author{Jochen Gl\"uck}
\address{Faculty of Computer Science and Mathematics, University of
  Passau, Germany}
\curraddr{}
\email{jochen.glueck@uni-passau.de}

\author{Andrii Mironchenko}
\address{Faculty of Computer Science and Mathematics, University of
  Passau, Germany}
\curraddr{}
\email{andrii.mironchenko@uni-passau.de}


\date{\today}
\begin{abstract}
	We prove new characterisations of exponential stability for positive linear discrete-time systems in ordered Banach spaces, in terms of small-gain conditions. Such conditions have played an important role in the finite-dimensional systems theory, but are relatively unexplored in the infinite-dimen\-sional setting, yet. 
	
	Our results are applicable to discrete-time systems in ordered Banach spaces that have a normal and generating positive cone. Moreover, we show that our stability criteria can be considerably simplified if the cone has non-empty interior or if the operator under consideration is quasi-compact.

	To place our results into context we include an overview of known stability criteria for linear (and not necessarily positive) operators and provide full proofs for several folklore characterizations from this domain.
\end{abstract}

\maketitle

\section{Introduction} \label{section:introduction}

Positive systems occur frequently in the modeling, analysis and control of dynamical systems, for instance, in chemical engineering, compartmental systems and ecological systems \cite{FaR00}.
Besides being interesting in their own right, positive systems have been instrumental in establishing stability properties of control systems, which are not positive per se.
In particular, within the small-gain approach \cite{DRW07,DRW10}, the stability criteria for large-scale interconnected systems are given in terms of stability of certain discrete-time monotone systems.
The guiding lights for these results are the well-known characterizations of exponential stability of finite-dimensional linear discrete-time systems in terms of \q{small-gain} and \q{no-increase} conditions, which are summarized, e.g., in \cite[Lemma 2.0.1]{Rue07}. A decisive tool for the proof of such criteria is the celebrated Perron--Frobenius theorem. 

For infinite-dimensional systems, set in the realm of ordered Banach spaces, the situation is more complex.
Stability of positive and non-positive linear operators and semigroups has already been studied for many decades, as is documented for instance in the monographs \cite{vanNeerven1996, Emelyanov2007, Eisner2010}, as well as in the chapters about the asymptotic behaviour of $C_0$-semigroups in \cite{Nagel1986, EngelNagel2000, BatkaiKramarRhandi2017}. Yet, many stability criteria that turned out to be useful guidelines in the finite-dimensional path from the linear to the non-linear case -- in particular, so-called small-gain theorems and non-increasing conditions -- have only been partially explored in the infinite-dimensional setting so far (although see \cite[Theorem~3.13 on p.\,120 and Theorem~4.6 on p.\,130]{Kra64}, where no-increase conditions have been used for the study of fixed points of monotone operators). This has several reasons.
On the one hand, the Krein--Rutman theorem, which is a (partial) infinite-dimensional extension of the Perron--Frobenius theorem, requires the operator under assumption to be quasi-compact, which is a rather strong assumption.
On the other hand, many finite-dimensional stability notions exhibit -- due to the compactness of the closed unit ball -- some kind of intrinsic uniformity, which is essential to characterize properties such as exponential stability. Lack of this uniformity in the infinite-dimensional setting breaks most of the known finite-dimensional criteria.

\subsection*{Contribution}

In this paper, we study discrete-time systems, playing a prominent role in the modeling, analysis, control and numerics of dynamical systems \cite{Aga00, Las12}. 
In Section~\ref{subsection:positive-discrete:stability}, we introduce several novel stability properties, most notably the uniform small-gain condition, and characterize the exponential stability of positive linear discrete-time systems in terms of such properties (Theorem~\ref{subsection:positive-discrete:stability}). Our assumptions on the ordered Banach space are not particularly restrictive; we merely assume that the cone is normal and generating, which is satisfied for many important classes of spaces.
Some, though not all, of the equivalences in Theorem~\ref{subsection:positive-discrete:stability} have been shown in \cite{MKG20} to hold even in the non-linear case; these results are related to the input-to-state stability of control systems with inputs and to so-called small-gain theorems.

In Section~\ref{subsection:positive-discrete:interior-points} we devote special attention to ordered Banach spaces whose cone has non-empty interior.
Subeigenvectors of operators inducing discrete-time positive exponentially stable systems are a key to the construction of Lyapunov functions for networks of stable systems, see \cite{DRW10} for the case of finite networks. Our results in Section~\ref{subsection:positive-discrete:interior-points} pave the way to extending these results to infinite networks; some results of this type have already been put to use \cite{KMS19}.

In Section~\ref{subsection:positive-discrete:quasi-compact} we treat systems given by quasi-compact operators.
By using the Krein--Rutman theorem for these operators, we show that the assumption of quasi-compactness allows us to extend most of finite-dimensional results to the infinite-dimensional setting.

To provide our results on positive systems with sufficient context, we include a small survey part in this article which is given in Section~\ref{section:discrete-time-tour}. There we discuss several results on uniform stability of (non-positive) linear systems which are scattered throughout the literature; we also include various references which point the reader to further related results.

Of particular importance is Theorem~\ref{thm:discrite-time-convolution} which collects several characterizations of uniform stability which are probably known to experts in operator theory or in infinite-dimensional system theory, but which we could not find in an explicit form in the literature. 

This paper is intended both for functional analysts working in the field of positive operators and systems, as well as for experts in dynamical systems and control theory. To make the paper easier accessible for all interested readers we discuss all relevant notions and properties from the theory of ordered Banach spaces in Section~\ref{section:ordered-banach-spaces}.

Moreover, to give those readers who are unfamiliar with the details of ordered Banach spaces a good intuition for the various concepts that we use, we discuss of several important examples of ordered Banach spaces.

We concentrate ourselves solely on linear systems. For a survey on the theory of nonlinear monotone discrete-time infinite-dimensional systems we refer to \cite{HiS05} and \cite[Section 5]{HiS05b}. 
Some early results have been presented in \cite[Chapter 1]{Hes91}. 

Nonlinear versions of small-gain type conditions studied in this paper have been applied for the analysis of nonlinear systems in \cite{MKG20}. 
In \cite{GlM21} characterizations for the negativity of the spectral bound of resolvent positive operators have been studied, which are continuous-time counterparts of the results studied in this paper.
Some of the results in Section~\ref{subsection:positive-discrete:interior-points} have been extended to the case of homogeneous and subadditive operators in \cite{MNK21} and applied for the construction of ISS Lyapunov functions for infinite networks of input-to-state stable systems with homogeneous and subadditive gain operators.
Nevertheless, in the nonlinear case many problems remain open.

\subsection*{Notation}

We use the conventions $\N = \{1,2,3,\dots\}$ and $\Z_+ = \{0,1,2,\dots\}$.

If $X$ is a Banach space, we denote the space of bounded linear operators on $X$ by $\calL(X)$ and we denote the dual space -- i.e., the space of bounded linear functionals on $X$ -- by $X'$. For $x' \in X'$ and $x \in X$ we use the common notation $\langle x', x \rangle := x'(x)$. The identity operator on a Banach space will be denoted by $\id$ (if the space is clear from the context).

If the underlying scalar field of the Banach space $X$ is complex, we denote the \emph{spectrum} an operator $T \in \calL(X)$ by $\sigma(T)$. The \emph{spectral radius} of $T$ is denoted by
\begin{align*}
	\spr(T) := \sup \{\modulus{\lambda}: \, \lambda \in \sigma(T)\} \in [0,\infty).
\end{align*}
If $\lambda \in \C$ is located in the complement of the spectrum (i.e., in the so-called \emph{resolvent set} of $T$), the operator $(\lambda \id - T)^{-1}$ is called the \emph{resolvent} of $T$ at $\lambda$.

For bounded linear operators on real Banach spaces, spectral properties are defined by means of \emph{complexification}; for details, we refer to the subsection on complexifications at the end of Section~\ref{section:ordered-banach-spaces}.

\section{Setting the stage: ordered Banach spaces and positive operators}
\label{section:ordered-banach-spaces}

In this section, we recall some background information on ordered Banach spaces that will be needed throughout the article.

\subsection*{Ordered Banach spaces} 
By an \emph{ordered Banach space} we mean a pair $(X,X^+)$ where $X$ is a real Banach space, and $X^+ \subseteq X$ is a non-empty closed set such that $\alpha X^+ + \beta X^+ \subseteq X^+$ for all scalars $\alpha, \beta \ge 0$ and such that $X^+ \cap (-X^+) = \{0\}$. The set $X^+$ is called the \emph{positive cone} in $X$.

The positive cone of an ordered Banach space $(X,X^+)$ induces a partial order $\le$ on $X$ which is given by $x \le y$ if and only if $y-x \in X^+$. In particular, $x \ge 0$ if and only if $x \in X^+$. The partial order $\le$ is compatible with addition and with multiplication by scalars $\alpha \ge 0$.

For any two vectors $x,z$ in an ordered Banach space we call the set $[x,z] := \{y \in X: \, x \le y \le z\}$ the \emph{order interval} between $x$ and $z$; this order interval is non-empty if and only if $x \le z$.

\subsection*{Generating and normal cones} 
Let $(X,X^+)$ be an ordered Banach space. The cone $X^+$ is called \emph{total} (or \emph{spatial}) if the vector subspace $X^+ - X^+ = \{x-y:x,y\in X^+\}$ of $X$ is dense in $X$. The cone is called \emph{generating} (or the space $X$ is called \emph{directed}) if we even have $X^+ - X^+ = X$. In other words, the cone is generating if and only if each vector $x \in X$ can be decomposed as $x = y-z$ for two vectors $y,z \in X^+$. In fact, this decomposition can always be done in a way that controls the norms of $y$ and $z$: 

If the cone $X^+$ is generating, then there exists a number $M > 0$ with the following property: for each $x \in X$ there exist $y,z \in X^+$ such that
\begin{align}
	\label{eq:bounded-decomposition}
	x = y-z \qquad \text{and} \qquad \norm{y}, \norm{z} \le M \norm{x};
\end{align}
see for instance \cite[Theorem~2.37(1) and~(3)]{AliprantisTourky2007}. The cone $X^+$ is called \emph{normal} if there exists a number $C > 0$ such that we have
\begin{align}
	\label{eq:normality-constant}
	\norm{x} \le C \norm{y} \qquad \text{whenever } 0 \le x \le y
\end{align}
in $X$. The cone is normal if and only if there exists an equivalent norm $\norm{\argument}'$ on $X$ which is \emph{monotone} in the sense that $\norm{x}' \le \norm{y}'$ whenever $0 \le x \le y$; see for instance \cite[Theorem~2.38(1) and~(2)]{AliprantisTourky2007}.

Finally, for each set $S\subset X$ denote by $\intt(S)$ the \emph{topological interior} of $S$. If $\intt(X^+)\neq\emptyset$, then we say that the cone $X^+$ has a non-empty interior. We shall have more to say about cones with non-empty interior in Section~\ref{subsection:positive-discrete:interior-points}, in particular in Proposition~\ref{prop:order-units}. Here we only mention that a cone with non-empty interior is automatically generating (Proposition~\ref{prop:order-units}(\ref{prop:order-units:itm:interior-point}) and~(\ref{prop:order-units:itm:order-unit-variant})).

\subsection*{Banach lattices}

A particularly well-behaved class of ordered Banach spaces is the class of so-called \emph{Banach lattices} \cite{Schaefer1974, Meyer-Nieberg1991}. An ordered Banach space $(X,X^+)$ is said to be \emph{lattice-ordered} if each two elements $x,y \in X$ have a smallest upper bound -- i.e., a so-called \emph{supremum} -- in $X$, which we then denote by $x \lor y$. This allows us to define the \emph{positive part}, the \emph{negative part} and the \emph{modulus} of a vector $x \in X$ by means of the formulae
\begin{align*}
	x^+  := x \lor 0, \qquad
	x^-  := (-x) \lor 0, \qquad
	\modulus{x}  := x \lor (-x).
\end{align*}
The cone in a lattice-ordered Banach space is always generating since we have $x = x^+ - x^-$ for each $x \in X$.

A \emph{Banach lattice} is a lattice-ordered Banach space which satisfies an additional compatibility assumption between norm and lattice operations, namely
\begin{align}
	\label{eq:compatibility-in-banach-lattices}
	\modulus{x} \le \modulus{y} \quad \Rightarrow \quad \norm{x} \le \norm{y}
\end{align}
for all $x,y \in X$. This implies, in particular, that the cone in a Banach lattice is always normal.

\subsection*{The distance to the cone}

For a subset $S$ and a vector $x$ in a Banach space $X$, we denote by
\begin{align*}
	\dist(x,S) := \inf\left\{ \norm{x-y}: \, y \in S \right\}
\end{align*}
the distance from $x$ to $S$. If $(X,X^+)$ is an ordered Banach space, we will often need the distance of points to the positive cone $X^+$. Due to the specific properties of cones, the distance function $\dist(\argument,X^+)$ is quite nicely behaved; in particular, it is not difficult to see that we have for all $x,y \in X$ and all $\alpha \in [0,\infty)$
\begin{align*}
	& \dist(x+y,X^+) \le \dist(x,X^+) + \dist(y,X^+), \\
	\text{and} \qquad & \dist(\alpha x,X^+) = \alpha \dist(x,X^+).
\end{align*}

\subsection*{Examples of ordered Banach spaces}

Here we recall a few examples of ordered Banach spaces for the convenience of readers that are not too familiar with this concept. Classical sequence and function spaces constitute several important classes of ordered Banach spaces, as explained in the following three examples.

\begin{examples}[Sequence spaces]
	\label{ex:sequence-spaces}
	\begin{enumerate}[(a)]
		\item\label{ex:sequence-spaces:itm:ell_p} 
Let $X = \ell_p := \{x \in \R^{\N}: \, \|x\|_{\ell_p} <\infty\}$ for $p \in [1,\infty]$, with the norms  $x = (x_n)_{n \in \N}\mapsto \|x\|_{\ell_p} := \Big(\sum_{n=1}^\infty |x_n|^p\Big)^{1/p}$ for $p<\infty$ and $x\mapsto \|x\|_{\ell_\infty}:=\sup_{n=1}^\infty |x_n|$.
				We endow $\ell_p$ spaces with the cone		
		\begin{align*}
			\ell_p^+ := \{(x_n)_{n \in \Z_+} \in \ell_p: \, x_n \ge 0 \text{ for all } n \in \N\}.
		\end{align*}
		Then $(\ell_p,\ell_p^+)$ is an ordered Banach space, and in fact even a Banach lattice (where the lattice operations can be performed entrywise); in particular, the cone $\ell_p^+$ is generating and normal. If $p = \infty$, then the cone has non-empty interior, whereas $\intt(\ell_p^+)=\emptyset$ for $p \in [1,\infty)$.
		
		\item\label{ex:sequence-spaces:itm:c-and-c_0} Let $X = c$ (the space of convergent real sequences) or $X = c_0$ (the space of real sequences that converge to $0$), endowed with the supremum norm and with the cone $X^+$ which is defined in the same way as in Example~\eqref{ex:sequence-spaces:itm:ell_p}. Then $(X,X^+)$ is an ordered Banach space and even a Banach lattice (with entrywise lattice operations) and thus, the cone $X^+$ is generating and normal. The interior of the cone is non-empty if $X = c$ and empty if $X = c_0$.
	\end{enumerate}
\end{examples}

\begin{example}[Spaces of integrable functions]
	\label{ex:integrable-functions}
	If $(\Omega,\mu)$ is an arbitrary measure space and $p \in [1,\infty]$, then the space $L^p(\Omega,\mu)$ over $\R$, endowed with the $p$-norm and the cone of all functions that are $\ge 0$ almost everywhere, is an ordered Banach space and actually even a Banach lattice (where the lattice operations are computed pointwise almost everywhere); hence, the cone is generating and normal.
	
	The interior of the cone is non-empty if $p=\infty$ or if $L^p(\Omega,\mu)$ is finite-dimensional; in all other cases, it is empty.
\end{example}

\begin{example}[Spaces of continuous functions]
	\label{ex:continuous-functions}
	Here are two more examples of Banach lattices (and thus, in particular, of ordered Banach spaces with generating and normal cone):
	
	If $\Omega$ is a topological space, then the space $C_b(\Omega)$ of all real-valued and bounded continuous functions is Banach lattice when endowed with the supremum norm and the cone of all those functions in $C_b(\Omega)$ that are $\ge 0$ everywhere on $\Omega$. The cone in this space has non-empty interior.
	
	Similarly, if $L$ is a locally compact Hausdorff space and $C_0(L)$ denotes that closed subspace of $C_b(L)$ consisting of functions that vanish at infinity, then $C_0(L)$ is a Banach lattice with respect to the norm and cone inherited from $C_b(L)$. The interior of the cone in $C_0(L)$ is non-empty if and only if $L$ is compact (in which we have $C_0(L) = C_b(L) = C(L)$, where $C(L)$ denote the space of all continuous real-valued functions on $L$).
\end{example}

A classical example of an ordered Banach space that is not a Banach lattice is the self-adjoint part of a non-commutative $C^*$-algebra (see for instance \cite[Examples~2.15 and~2.16]{GlueckWeber2020} for some additional information about the order structure of such spaces). Non-normal cones can, for example, be found in various spaces of continuously differentiable functions and in Sobolev spaces; see, e.g., \cite[Examples~2.17 and~2.18]{GlueckWeber2020}.

Finally, we give a simple example, taken from \cite[pp.\,35--36]{KrasnoselskiiLifshitsSobolev1989}, of an ordered Banach space with total but non-generating cone:

\begin{examples}[A space with non-generating cone]
	We consider again the sequence space $c_0$ with the supremum norm, but in contrast to Example~\ref{ex:sequence-spaces}(\ref{ex:sequence-spaces:itm:c-and-c_0}) we now set
	\begin{align*}
		c_0^+ := \left\{ x \in c_0: \; x_0 \ge \left(\sum\nolimits_{k=1}^\infty x_k^2\right)^{1/2} \right\}.
	\end{align*}
	Note that the series on the right might be infinite, in which case the inequality is not satisfied. The cone $c_0^+$ is closed in $c_0$ by Fatou's lemma. 
	The span of the cone is equal to $\ell_2$, so the cone is total but not generating in $c_0$. Moreover, it is not difficult to show that the cone is normal.

	In contrast to the situation described above, the space $\ell_2$ with the same cone -- which can also be written as 
	\begin{align*}
		\left\{ x \in \ell_2: \; x_0 \ge \left(\sum\nolimits_{k=1}^\infty x_k^2\right)^{1/2} \right\}
	\end{align*}
	and which is sometimes called the \emph{Lorentzian cone} -- is an ordered Banach space whose cone has non-empty interior since, for instance, the vector $(1,0,0,\ldots)$ is an interior point.
\end{examples}

\subsection*{Positive operators} Let $(X,X^+)$ and $(Y,Y^+)$ be ordered Banach spaces. A linear mapping $A: X \to Y$ is called \emph{positive}, which we denote as $A \ge 0$, if $AX^+ \subseteq Y^+$. A linear mapping $A$ is positive if and only if it respects the order relation (i.e.\ $Ax_1 \le Ax_2$ whenever $x_1 \le x_2$).

We are particularly interested in \emph{bounded} linear operators. Interestingly though, this assumption is often redundant: if the cone $X^+$ is generating, then every positive linear operator $A: X \to Y$ is automatically bounded \cite[Theorem~2.32]{AliprantisTourky2007}.

\subsection*{Duality of ordered Banach spaces} 
Let $(X, X^+)$ be an ordered Banach space. The subset
\begin{align*}
	(X')^+ := \{x' \in X': \, \langle x', x \rangle \ge 0 \text{ for all } x \in X^+\}
\end{align*}
of the dual space $X'$ is called the \emph{dual wedge} of $X^+$. The elements of $(X')^+$ are called the \emph{positive functionals} on $X$.

The dual wedge is also closed (even weak${}^*$-closed), convex and invariant with respect to multiplication by non-negative scalars. The dual wedge $(X')^+$ is a cone -- i.e., its intersection with $-(X')^+$  is $\{0\}$ -- if and only if the cone $X^+$ is total in $X$.

Let $(X,X^+)$ be an ordered Banach space with total cone $X^+$; as explained in the previous paragraph, the dual space $(X', (X')^+)$ is then an ordered Banach space, too. The properties of $X^+$ and $(X')^+$ are closely related. For instance, $X^+$ is generating if and only if $(X')^+$ is normal \cite[Theorem~4.6]{KrasnoselskiiLifshitsSobolev1989}; and conversely, $X^+$ is normal if and only if $(X')^+$ is generating \cite[Theorem~4.5]{KrasnoselskiiLifshitsSobolev1989}.

\subsection*{Complexifications} Since we are going to use the spectral theory of positive operators, a word on \emph{complexifications} is in order. The underlying scalar field of an ordered Banach space $(X,X^+)$ is real; but every real Banach space $X$ has a (in general, non-unique) \emph{complexification} $X_\C$, which is a Banach space over $\C$, and each bounded linear operator $T$ between real Banach spaces $X$ and $Y$ can be uniquely extended to a bounded $\C$-linear operator $T_\C: X_\C \to Y_\C$; moreover, this extension always satisfies the norm estimate $\norm{T} \le \norm{T_\C} \le 2 \norm{T}$. Whenever we talk about spectral properties of $T$, we tacitly mean the corresponding spectral properties of the complex extension $T_\C$.
For an overview of complexifications of Banach spaces, we refer for instance to \cite{MunozSarantopoulosTonge1999} and \cite[Appendix~C]{Glueck2016}.

\section{Stability for positive systems}
\label{section:positive-discrete-time-systems}

After a brief description of ordered Banach spaces and positivity, we now focus on the stability analysis of positive linear discrete-time evolution equations.

For a Banach space $X$ and an operator $T \in \calL(X)$, consider the \emph{discrete-time system} induced by $T$,
\begin{align}
	\label{eq:Gamma-discrete-time-system}
	x(k+1) = T x(k) \quad \text{for all } k\in\Z_+.
\end{align}
We are interested in the question whether the solutions to this system converge uniformly to $0$ as $k \to \infty$; this is made precise in the following definition.

\begin{definition}
	\label{def:ISS-UGAS-UGES_discrete_time}
	 For a Banach space $X$ and an operator $T \in \calL(X)$, the system \eqref{eq:Gamma-discrete-time-system} is called
	\begin{enumerate}[(a)]
		\item\label{itm:def::ISS-UGAS-UGES_discrete_time-1}
		\emph{uniformly asymptotically stable}, if there is a sequence of real numbers $0 \le a_k \to 0$ such that
		\begin{align*}
			\norm{T^k x} \le a_k \norm{x} \quad \text{for all } x \in X,\ k\in\Z_+;
		\end{align*}
		
		\item\label{itm:def::ISS-UGAS-UGES_discrete_time-2}
		\emph{uniformly exponentially stable} (or \emph{uniformly power stable}), if there exist real numbers $a \in [0,1)$ and $M>0$ such that
		\begin{align*}
			\norm{T^k x} \leq M a^{k} \norm{x} \quad \text{for all } x \in X \text{ and all } k\in\Z_+;
		\end{align*}
		
		\item\label{itm:def::ISS-UGAS-UGES_discrete_time-3} 
		\emph{uniformly weakly attractive} 
		if, for each $r>0$ and each $\varepsilon>0$, there is a time $\tau$ with the following property: 
		for each $x \in X$ of norm $\norm{x} \leq r$ there is $k\leq \tau$ such that $\norm{T^k x} \leq \varepsilon$.
	\end{enumerate}
\end{definition}

The stability notions \eqref{itm:def::ISS-UGAS-UGES_discrete_time-1} and \eqref{itm:def::ISS-UGAS-UGES_discrete_time-2} from the previous definition can be expressed in terms of operator norms. Indeed, the system~\eqref{eq:Gamma-discrete-time-system} is uniformly asymptotically stable if and only if $\norm{T^k} \to 0$ as $k \to \infty$; and the system~\eqref{eq:Gamma-discrete-time-system} is uniformly exponentially stable if and only if there exist numbers $M > 0$ and $a \in [0,1)$ such that $\norm{T^k} \le M a^k$ holds for all $k \in \Z_+$.

The notion of uniform weak attractivity was introduced in \cite{Mir17a}, motivated by the classical notion of weak attractivity \cite{Bha66, BhS02}. We stress that in some works \q{weak} stability concepts have another meaning, namely  convergence properties in the weak topology on $X$, see, e.g., \cite{Eisner2010}.

For a Banach space $X$ and an operator $T\in \calL(X)$ we denote the spectral radius of (the complexification of) $T$ by $\spr(T)$.
A well-known result in the stability theory of discrete-time systems \eqref{eq:Gamma-discrete-time-system} is the following:

\begin{proposition}
	\label{prop:UGAS_discrete_systems_and_Small-gain_theorem_Operators_with_Order-uniform-conditions}
	Let $X$ be a Banach space and $T \in \calL(X)$. 
	The following assertions are equivalent:
	\begin{enumerate}[\upshape (i)]
		\item \label{itm:Ulin-SGC-criteria-1} We have $\spr(T) < 1$.
		\item\label{itm:Ulin-SGC-criteria-2} The system \eqref{eq:Gamma-discrete-time-system} is uniformly asymptotically stable.
		\item\label{itm:Ulin-SGC-criteria-3} The system \eqref{eq:Gamma-discrete-time-system} is uniformly exponentially stable.
		\item\label{itm:Ulin-SGC-criteria-32} The system \eqref{eq:Gamma-discrete-time-system} is uniformly weakly attractive.
	\end{enumerate}
\end{proposition}

\begin{proof}
	For \eqref{itm:Ulin-SGC-criteria-1} $\Iff$ \eqref{itm:Ulin-SGC-criteria-2} $\Iff$ \eqref{itm:Ulin-SGC-criteria-3} see,
e.g.,~\cite[Lemma 2.1 and Theorem 2.2]{Prz88} and for 
	\eqref{itm:Ulin-SGC-criteria-2} $\Iff$ \eqref{itm:Ulin-SGC-criteria-32} see \cite[Proposition 5.1]{Mir17a} (shown for continuous-time systems, but easily transferable to the discrete-time case).
\end{proof}

\subsection{General stability criteria for positive systems}
\label{subsection:positive-discrete:stability}

From now on we focus in this section on the positive case, i.e., we study linear operators on \emph{ordered Banach spaces} and we assume that the operators are \emph{positive} in the sense that they leave the positive cone invariant. 

Positivity of the operator $T$ does not simplify the criteria for uniform exponential stability considered in Proposition~\ref{prop:UGAS_discrete_systems_and_Small-gain_theorem_Operators_with_Order-uniform-conditions}. On the other hand though, positivity allows to show diverse characterizations of a very different nature, which are typically referred to as \emph{small-gain} type conditions.
We proceed to our main result for discrete-time positive linear systems, which contains several such criteria.

\begin{theorem} 
	\label{thm:stability-for-pos-ops}
	Let $(X,X^+)$ be an ordered Banach space with generating and normal cone and let $T \in \calL(X)$ be positive. Then the following assertions are equivalent:
	\begin{enumerate}[\upshape (i)]
		\item\label{thm:stability-for-pos-ops:itm:stability} 
		\emph{Uniform exponential stability:} 
		The system~\eqref{eq:Gamma-discrete-time-system} satisfies the equivalent criteria of 
		Proposition~\ref{prop:UGAS_discrete_systems_and_Small-gain_theorem_Operators_with_Order-uniform-conditions}, 
		i.e., we have $\spr(T) < 1$.
		
		\item\label{thm:stability-for-pos-ops:itm:pos-resolvent} 
		\emph{Positivity of the resolvent at $1$:} 
		The operator $\id - T: X \to X$ is bijective and its inverse $(\id - T)^{-1}$ is positive.\footnote{Note that as $\id - T$ is invertible and bounded, then $(\id - T)^{-1}$ is closed, and since $T$ is surjective, $(\id - T)^{-1}$ is bounded by a closed graph theorem. Thus, $1 \in \rho(T)$ and $(\id - T)^{-1}$ is indeed a resolvent.}
		
		\item\label{thm:stability-for-pos-ops:itm:mbi} 
		\emph{Monotone bounded invertibility property:} 
		There exists a number $c \ge 0$ such that
		\begin{align*}
			(\id - T)x \le y \qquad \Rightarrow \qquad \norm{x} \le c \norm{y}
		\end{align*}
		holds for all $x,y \in X^+$.
				
		\item\label{thm:stability-for-pos-ops:itm:uniform-small-gain} 
		\emph{Uniform small-gain condition:} 
		There is a number $\eta > 0$ such that
		\begin{align*}
			\dist\big((T-\id)x,X^+\big) \ge \eta \norm{x}
		\end{align*}
		for each $x \in X^+$.
			
		\item\label{thm:stability-for-pos-ops:itm:perturbed-small-gain} 
		\emph{Robust small-gain condition:} 
		There exists a number $\varepsilon > 0$ such that
		\begin{align}
			\label{eq:Robust small-gain condition}
			(T+P)x \not\ge x
		\end{align}
		for every $0 \not= x \in X^+$ and for every positive operator $P \in \calL(X)$ of norm $\norm{P} \le \varepsilon$.
		
		\item\label{thm:stability-for-pos-ops:itm:perturbed-rank-1-small-gain} 
		\emph{Rank-$1$ robust small-gain condition:} 
		There exists a number $\varepsilon > 0$ such that 
		\begin{align*}
			(T+P)x \not\ge x
		\end{align*}
		for every $0 \not= x \in X^+$ and for every positive operator $P \in \calL(X)$ of rank $1$ and of norm $\norm{P} \le \varepsilon$.
	\end{enumerate}
\end{theorem}

\begin{remarks}
	\begin{enumerate}[\upshape (a)]
		\item 
		The terminology \q{small-gain condition} stems from the study of interconnected systems in systems and control theory.
		In this context, the gain describes the response of the system on the applied input.
		As an example, consider two systems $\Sigma_1$ and $\Sigma_2$. 
		If $\gamma_{12}>0$ is the gain describing the influence of the system $\Sigma_2$ on the system $\Sigma_1$, and 
		$\gamma_{21}>0$ is the gain describing the influence of the system $\Sigma_1$ on the $\Sigma_2$, then the condition $\gamma_{12}\cdot\gamma_{21}<1$ guarantees in a proper context the stability of the feedback interconnection of $\Sigma_1$ and $\Sigma_2$ and is referred to as a \q{small-gain condition}.
		At the same time, one has the equivalence
		\[
			\quad
			\gamma_{12}\cdot\gamma_{21}<1 
			\; \qiq \;
			\begin{pmatrix}
				0 &  \gamma_{12} \\
				\gamma_{21}	& 0
			\end{pmatrix} 
			\begin{pmatrix}
				s_1\\
				s_2	
			\end{pmatrix}
			\not\geq
			\begin{pmatrix}
				s_1\\
				s_2	
			\end{pmatrix}
			\; \text{ for all } (s_1,s_2) \in \R^2_+\backslash\{0\}
		\]
		(for the implication \q{$\Leftarrow$}, just take $s_1:=1$, $s_2:=\gamma_{21}$).
		This explains the use of the term \q{small-gain condition} for conditions like \eqref{eq:Robust small-gain condition}.
		
		\item 
		In view of Theorem~\ref{thm:stability-for-pos-ops}(\ref{thm:stability-for-pos-ops:itm:uniform-small-gain}) we point out 
		that, in the important case where $X$ is a Banach lattice, the distance to the positive cone can be computed by means of the formula
		\begin{align*}
			\dist(Tx-x,X^+) = \norm{\left(Tx-x\right)^-}
		\end{align*}
		for each vector $x$.
		This follows easily from standard properties of Banach lattices; 
		the argument can be found in detail in \cite[Remark~7.3]{MKG20}.
		
		\item 
		In assertion~(\ref{thm:stability-for-pos-ops:itm:perturbed-small-gain})
		of the theorem, it does not suffice to consider only a single non-zero operator $P$ as a perturbation. As a counterexample, let $p \in [1,\infty]$ and let $T: \ell_p \to \ell_p$ denote the right shift operator given by
		\begin{align*}
			T: (x_1, x_2, \dots) \mapsto (0, x_1, x_2, \dots).
		\end{align*}
		Moreover, let $P = \frac{1}{2}\id: \ell_p \to \ell_p$ denote half the identity operator. 
		Then it is easy to check that $(T + P)x \not\ge x$ for each non-zero vector $x \ge 0$.
		Yet, $T$ has spectral radius $1$.
		
		A related observation is made in Example~\ref{ex:SS_does_not_imply_ExpSt}.
		
		\item 
		We note that all the equivalent conditions in Theorem~\ref{thm:stability-for-pos-ops} 
		can also be formulated in terms of the dual operator $T'$, 
		since the dual operator is positive, too, and since $\spr(T') = \spr(T)$.
	\end{enumerate}
\end{remarks}

A couple of results that are loosely related to the small-gain type conditions above can be found in Sections~6 and~7 of the classical paper \cite{Karlin1959}.
For the proof of Theorem~\ref{thm:stability-for-pos-ops} we need two lemmas. The first one is about so-called \emph{approximate eigenvectors}. A number $\lambda$ is called an \emph{approximate eigenvalue} of the operator $T$ if there exists a sequence of vectors $(x_n)_{n\in\N}$ (in the complexification of $X$) such that $\norm{x_n} = 1$ for each $n$ and such that $\lambda x_n - Tx_n \to 0$; in this case, the sequence $(x_n)_{n\in\N}$ is called an \emph{approximate eigenvector} that corresponds to the approximate eigenvalue $\lambda$. It is easy to see that every approximate eigenvalue of $T$ is a spectral value of $T$ (i.e., belongs to the spectrum $\sigma(T)$).

\begin{lemma}
	\label{lem:pos-approx-ev}
	Let $(X,X^+)$ be an ordered Banach space with generating and normal cone and let $T \in \calL(X)$ be positive. 
	Then $\spr(T)$ is an approximate eigenvalue of $T$, and there exists an associated approximate eigenvector $(x_n)_{n\in\N}$ that consists of positive vectors in $X$.
\end{lemma}

The proof of the lemma is a simple variation of a standard argument in operator theory. For the convenience of the reader and in order to be more self-contained, we include the details.

\begin{proof}[Proof of Lemma~\ref{lem:pos-approx-ev}]
	Since $X^+$ is generating and normal, the positivity of $T$ implies that $\spr(T)$ is a spectral value of $T$ (see for instance \cite[paragraph~2.2 on p.\ 311]{SchaeferWolff1999}). Thus, we have $\norm{(r\id-T)^{-1}} \to \infty$ as $r \downarrow \spr(T)$. By the uniform boundedness theorem, we can find a vector $u \in X$ and a sequence $(r_n)_{n\in\N}$ in $ (\spr(T),\infty)$ that converges to $\spr(T)$ such that $\norm{(r_n\id-T)^{-1}u} \to \infty$ as $n\to\infty$. As the cone $X^+$ is generating, we can write $u$ as $u = v-w$ for vector $v,w \in X^+$; consequently, at least one of the sequences $\big((r_n\id - T)^{-1}v\big)_{n \in \N}$ and $\big((r_n\id - T)^{-1}w\big)_{n \in \N}$ is unbounded. Thus, by choosing a subsequence of $(r_n)_{n \in \N}$ and by interchanging $v$ and $w$ if necessary, we may assume that $\alpha_n := \norm{(r_n\id-T)^{-1}v} \to \infty$. 
	
	Now define $x_n := \frac{1}{\alpha_n} (r_n\id - T)^{-1}v$ for each $n \in \N$. Clearly, $\norm{x_n} = 1$. Moreover, the resolvent $(r_n\id - T)^{-1}$ is positive for each $n$ due to the Neumann series representation formula; thus, all vectors $x_n$ are positive. It only remains to show that $(\spr(T)\id - T)x_n \to 0$. And indeed, we have
	\begin{align*}
		(\spr(T)\id - T)x_n 
		& = (\spr(T) - r_n) x_n + (r_n \id - T) \frac{1}{\alpha_n} (r_n \id - T)^{-1} v \\
		& = (\spr(T) - r_n) x_n + \frac{v}{\alpha_n} \to 0,
	\end{align*}
	since $\alpha_n \to \infty$ as $n \to \infty$.
\end{proof}

The second lemma that we need for the proof of Theorem~\ref{thm:stability-for-pos-ops} is the following simple observation.

\begin{lemma}
	\label{lem:decomposition-in-dual-space}
	Let $(X,X^+)$ be an ordered Banach space with total and normal cone (such that the dual cone $(X')^+$ is generating in $X'$) and let $M' \ge 0$ be a real number such that the decomposition property~\eqref{eq:bounded-decomposition} holds in the dual space $X'$ for $M'$.
	
	For each vector $x \in X^+$ of norm $1$ there exists a functional $z' \in (X')^+$ of norm at most $M'$ such that
	\begin{align*}
		\langle z', x \rangle \ge 1.
	\end{align*}
\end{lemma}

\begin{proof}
	Take any $x\in X^+$ with $\norm{x}=1$. Due to the Hahn--Banach extension theorem, there exists a functional $x' \in X'$ of norm $1$ such that $\langle x', x\rangle = 1$ (see, e.g., \cite[p.\,636]{CuZ20}). 
	We can decompose $x'$ as
	\begin{align*}
		x' = z' - y',
	\end{align*}
	where $z'$ and $y'$ are elements of $(X')^+$ of norm at most $M'$. Then we have
	\begin{align*}
		\langle z',x \rangle \ge \langle z', x \rangle - \langle y', x\rangle = \langle x', x\rangle = 1,
	\end{align*}
	which is the assertion.
\end{proof}

\begin{proof}[Proof of Theorem~\ref{thm:stability-for-pos-ops}]
	\Implies{thm:stability-for-pos-ops:itm:stability}{thm:stability-for-pos-ops:itm:pos-resolvent} 
	If the $\spr(T) < 1$, then $\id-T$ is clearly invertible, and it follows from the Neumann series representation of the resolvent that
	\begin{align*}
		(\id-T)^{-1} = \sum_{k=0}^\infty T^k \ge 0;
	\end{align*}
	the inequality at the end follows by applying the operator series to vectors $x \in X^+$ and using that $X^+$ is closed.
	
	\Implies{thm:stability-for-pos-ops:itm:pos-resolvent}{thm:stability-for-pos-ops:itm:mbi} 
	Let $C \in [0,\infty)$ be the normality constant from inequality~\eqref{eq:normality-constant}. If $x,y \in X^+$ and $(\id-T)x \le y$ we obtain from the positivity of the resolvent $(\id-T)^{-1}$ that $x \le (\id-T)^{-1}y$, and hence
	\begin{align*}
		\norm{x} \le C \norm{(\id-T)^{-1}} \norm{y}.
	\end{align*}
	This proves the monotone bounded invertibility property with $c = C \norm{(\id-T)^{-1}}$.
	
	\Implies{thm:stability-for-pos-ops:itm:mbi}{thm:stability-for-pos-ops:itm:uniform-small-gain} 
	Let $c>0$ be as in~\eqref{thm:stability-for-pos-ops:itm:mbi}. Since the cone $X^+$ is assumed to be generating, we can find a constant $M >0$ as in the decomposition property~\eqref{eq:bounded-decomposition}. 
	
	Now fix $x \in X^+$ and let $\varepsilon > 0$ be arbitrary; we are going to show that
	\begin{align}
		\label{eq:proof-of-thm-stability-for-pos-ops}
		\dist\big((T-\id)x, X^+\big) \ge \frac{1}{cM}\norm{x} - \varepsilon.
	\end{align}
	
	For convenience, denote $a := (T-\id)x$. By the definition of the distance, we can find a vector $z \in X^+$ such that $\norm{a-z} \le \dist(a,X^+) + \varepsilon$, and we set $y := a-z$. Now we decompose the vector $y$ as
	\begin{align*}
		y = u-v,
	\end{align*}
	where $u,v$ are in $X^+$ and satisfy the norm estimate
	\begin{align*}
		\norm{u}, \norm{v} \le M \norm{y} \le M \dist(a,X^+) + M \varepsilon.
	\end{align*}
	For the vector $(\id-T)x$ we have the estimate
	\begin{align*}
		(\id-T)x = -a = -y-z = v-u-z \le v,
	\end{align*}
	so the monotone bounded invertibility property from~(\ref{thm:stability-for-pos-ops:itm:mbi}) implies that
	\begin{align*}
		\norm{x} \le c \norm{v} \le cM \big(\dist(a,X^+) + \varepsilon\big).
	\end{align*}
	So we have indeed shown the claimed inequality~\eqref{eq:proof-of-thm-stability-for-pos-ops}. Since $\varepsilon$ was arbitrary, this gives us the uniform small-gain condition with $\eta = \frac{1}{cM}$.
	
	\Implies{thm:stability-for-pos-ops:itm:uniform-small-gain}{thm:stability-for-pos-ops:itm:perturbed-small-gain} 
	Choose $\varepsilon = \eta/2$, where $\eta$ is the number from~\eqref{thm:stability-for-pos-ops:itm:uniform-small-gain}. Now, let $x$ be a non-zero element of $X^+$ and let $P \in \calL(X)$ be a positive linear operator of norm at most $\varepsilon$. We have to show that $Tx + Px \not\ge x$, and to this end we may -- and will, in order to simplify the notation -- assume that $x$ has norm $1$. 
	
	For each vector $z \in X^+$, we know from the uniform small-gain condition that $Tx-x$ has distance at least $\eta$ from $z$. Since $\norm{Px} \le \eta/2$, it follows that $Tx+Px-x$ still has distance at least $\eta/2$ from $z$, so
	\begin{align*}
		\dist(Tx+Px-x,X^+) \ge \frac{\eta}{2} > 0.
	\end{align*}
	In particular, $Tx+Px-x$ is not in the cone, so $Tx+Px \not\ge x$.
	
	\Implies{thm:stability-for-pos-ops:itm:perturbed-small-gain}{thm:stability-for-pos-ops:itm:perturbed-rank-1-small-gain} This implication is obvious.
	
	\Implies{thm:stability-for-pos-ops:itm:perturbed-rank-1-small-gain}{thm:stability-for-pos-ops:itm:stability} 
	Let $\varepsilon>0$ be as in~\eqref{thm:stability-for-pos-ops:itm:perturbed-rank-1-small-gain}.
	We argue by contraposition: assume that $\spr(T) \ge 1$. 
	
	By Lemma~\ref{lem:pos-approx-ev}, $\spr(T)$ is an approximate eigenvalue of $T$ and there exists a corresponding approximate eigenvector $(x_n)_{n \in \N}$ in $X^+$; more precisely, this means that each vector $x_n$ has norm $1$ and that
	\begin{align*}
		(T - \spr(T)\id)x_n \to 0.
	\end{align*}
	
	Since the cone in our space is generating, we can choose a number $M \in [0,\infty)$ as in the decomposition property~\eqref{eq:bounded-decomposition}. Since the dual cone $(X')^+$ in $X'$ is generating, too (due to the normality of $X^+$), there also exists a constant $M' \in [0,\infty)$ with the same property for the dual cone $(X')^+$.
	
	For each index $n$ we can decompose the vector $(T - \spr(T)\id)x_n$ as
	\begin{align*}
		(T - \spr(T)\id)x_n = y_n - z_n,
	\end{align*}
	where $y_n,z_n$ are vectors in $X^+$ of norm at most $M \norm{(T-\spr(T)\id)x_n}$. If we choose a sufficiently large number $n_0 \in \N$, we thus have $M'\norm{z_{n_0}} \le \varepsilon$.
	
	We now choose a functional $z' \in (X')^+$ of norm at most $M'$ such that $\langle z', x_{n_0}\rangle \ge 1$; such a functional $z'$ exists, as shown in Lemma~\ref{lem:decomposition-in-dual-space}. The rank-$1$ operator $P \in \calL(X)$ that is defined by the formula
	\begin{align*}
		Pv = \langle z', v\rangle z_{n_0} \quad \text{for each } v \in X,
	\end{align*}
	is positive and has norm
	\begin{align*}
		\norm{P} =\sup_{\norm{v}=1} \norm{ \langle z', v\rangle z_{n_0} }
		= \sup_{\norm{v}=1} \modulus{\langle z', v\rangle} \norm{ z_{n_0} } = \norm{z'} \norm{z_{n_0}} \le M' \norm{z_{n_0}} \le \varepsilon.
	\end{align*}
	On the other hand, we have
	\begin{align*}
		Tx_{n_0} + Px_{n_0} - x_{n_0} & \ge  (T - \spr(T)\id)x_{n_0} + Px_{n_0} \\
		& = y_{n_0} - z_{n_0} + \langle z', x_{n_0} \rangle z_{n_0} \ge 0
	\end{align*}
	since $\langle z', x_{n_0} \rangle \ge 1$. Hence, we have $Tx + Px \ge x$ for $x := x_{n_0}$.
\end{proof}

\begin{remark}[Nonlinearization of small-gain conditions]
	\label{rem:Nonlinear-uniform-SGC}	
	In \cite{MKG20} another definition of the uniform small-gain condition was used, namely that a nonlinear operator $T: X^+ \to X^+$ satisfies the \emph{uniform small-gain condition} if there is a homeomorphism $\eta: \R_+ \to \R_+$ (i.e., $\eta(0)=0$, $\eta$ is continuous and strictly increasing to infinity) such that 
	\begin{align*}
		\dist(Tx-x,X^+) \ge \eta (\norm{x}) \quad \text{for all } x \in X^+.
	\end{align*}
	For linear operators this definition coincides with our definition since the preceding inequality implies
	\begin{align*}
		\operatorname{dist}(Tx-x, X^+) 
		&= \inf_{y \in X^+} \norm{Tx-x-y} \\
		&= \norm{x} \inf_{y \in X^+} 
			\norm{ 
				\Big(T\frac{x}{\norm{x}}-\frac{x}{\norm{x}}\Big)-\frac{y}{\norm{x}} 
			} \\
		&= \norm{x} \operatorname{dist}
			\Big(
				\Big(T\frac{x}{\norm{x}}-\frac{x}{\norm{x}}\Big), X^+
			\Big) 
			\geq \eta(1)\norm{x}.
	\end{align*}
	for each $x \in X^+ \setminus \{0\}$.
	
	Similarly, one can \q{nonlinearize} the monotone bounded invertibility property.
	Furthermore, one can show the equivalence of the nonlinear uniform small-gain condition and a nonlinear version of 
	the monotone bounded invertibility property also for the case of nonlinear monotone operators, even if the cone $X^+$ is not normal, see \cite[Proposition 7.1]{MKG20}.
	As a consequence, we could obtain the equivalence of items \eqref{thm:stability-for-pos-ops:itm:mbi} and \eqref{thm:stability-for-pos-ops:itm:uniform-small-gain} in Theorem~\ref{thm:stability-for-pos-ops} from \cite[Proposition 7.1]{MKG20}.
	However, we preferred to give a direct and self-contained proof. \hfill $\qed$
\end{remark}

\begin{remark}
	\label{rem:Relation-to-DMS19a} 
	The motivation for characterizing uniform exponential stability of a positive system as in 
	items~\eqref{thm:stability-for-pos-ops:itm:perturbed-small-gain} 
	and \eqref{thm:stability-for-pos-ops:itm:perturbed-rank-1-small-gain} 
	of Theorem~\ref{thm:stability-for-pos-ops} is the robust strong small-gain condition 
	introduced in \cite{DMS19a} for the small-gain analysis of infinite networks.
\end{remark}

\subsection{Cones with interior points}
\label{subsection:positive-discrete:interior-points}

In this section we consider the special case where the cone $X^+$ in an ordered Banach space contains an interior point. To this end it is worthwhile to recall a few characterisations of interior points in the cone:

\begin{proposition} \label{prop:order-units}
	Let $(X,X^+)$ be an ordered Banach space and let $z \in X^+$. The following assertions are equivalent:
	\begin{enumerate}[\upshape (i)]
		\item\label{prop:order-units:itm:interior-point} The vector $z$ is an element of the topological interior of $X^+$.
		
		\item\label{prop:order-units:itm:order-unit} The vector $z$ is an \emph{order unit}, i.e., for each $y \in X$ there exists $\varepsilon > 0$ such that $z \ge \varepsilon y$.
		
		\item\label{prop:order-units:itm:order-unit-variant} The cone $X^+$ is generating, and for each $y \in X^+$ there exists $\varepsilon > 0$ such that $z \ge \varepsilon y$.
		
		\item\label{prop:order-units:itm:principal-ideal} The so-called \emph{principal ideal} $\bigcup_{n \in \N} [-nz,nz]$ equals $X$.
		
		\item\label{prop:order-units:itm:ball} There exists $\varepsilon > 0$ such that $z \ge y$ for every $y \in X$ of norm $\norm{y} \le \varepsilon$.
	\end{enumerate}
\end{proposition}
\begin{proof}
	These equivalences are standard results in the theory of ordered Banach spaces; see for instance \cite[Proposition~2.11]{GlueckWeber2020} for a reference where all these equivalences are proved in one place.
\end{proof}

We point out that we do not a priori assume the cone $X^+$ to have non-empty interior in Proposition~\ref{prop:order-units}. 

If $(X, X^+)$ is an ordered Banach space and the cone $X^+$ has non-empty interior, we will write $x \ll y$ (or, synonymously, $y \gg x$) if $y-x \in \intt(X^+)$.
We will need the following simple estimate for the spectral radius of a positive operator.

\begin{proposition}
	\label{prop:spr-super-eigenvector}
	Let $(X,X^+)$ be an ordered Banach space and let $T \in \calL(X)$ be positive. 
	Further assume that $\intt\, X^+ \neq \emptyset$. Then 
	\begin{align*}
		\spr(T) \geq 
		\inf \left\{ \lambda \in [0,\infty) : \; \exists z \in \intt(X^+) \mbox{ s.t. } Tz \ll \lambda z \right\}.
	\end{align*}
\end{proposition}

\begin{proof}
	Let $\lambda > \spr(T)$, and take an arbitrary point $y \in \intt(X^+)$. 
	Then the vector $z = z(y):= (\lambda \id - T)^{-1}y$ is also in $\intt(X^+)$.
	Indeed, it follows from the Neumann series representation of the resolvent 
	and from the positivity of $T$ that
	\begin{align}
	\label{eq:Expression-point-of-strict-decay}
		z = \sum_{k=0}^\infty \frac{T^k}{\lambda^{k+1}}y  \geq \frac{1}{\lambda}y;
	\end{align}
	hence, $z \in \intt(X^+)$ as $y \in \intt(X^+)$.
	Now, we estimate $Tz$ as
	\begin{align*}
		Tz
		& = \big(T - \lambda \id + \lambda \id\big) \, \big(\lambda \id - T\big)^{-1} y 
		 = - y + \lambda \big(\lambda \id - T\big)^{-1} y  = -y + \lambda z \ll \lambda z,
	\end{align*}
	which implies the claim.
\end{proof}

A result closely related to Proposition~\ref{prop:spr-super-eigenvector}, but for continuous-time systems, can be found in \cite[formula~(5.2) in Theorem~5.3]{ArendtChernoffKato1982}. 
A non-empty interior of the cone yields further powerful criteria for stability of positive systems; this is the content of the next theorem.

\begin{theorem}
	\label{thm:stability-for-pos-ops-interior-point}
	Let $(X,X^+)$ be an ordered Banach space with normal cone and suppose that the cone has non-empty interior. For every positive operator $T \in \calL(X)$, the following assertions are equivalent:
	\begin{enumerate}[\upshape (i)]
		\item\label{thm:stability-for-pos-ops-interior-point:item:stability} 
		\emph{Uniform exponential stability:}
		The system~\eqref{eq:Gamma-discrete-time-system} satisfies the equivalent criteria of 
		Proposition~\ref{prop:UGAS_discrete_systems_and_Small-gain_theorem_Operators_with_Order-uniform-conditions}, 
		i.e., we have $\spr(T) < 1$.
				
		\item\label{thm:stability-for-pos-ops-interior-point:item:small-gain-dual} 
		\emph{Dual small-gain condition:}
		For each $0 \not= x' \in (X')^+$ we have
		\begin{align*}
			T'x' \not\geq x'.
		\end{align*}
		
		\item\label{thm:stability-for-pos-ops-interior-point:item:small-gain-interior-point-all} 
		\emph{Interior point small-gain condition, first version:} For every interior point $z$ of $X^+$ there is a number $\eta > 0$ such that
		\begin{align*}
			Tx \not\ge x - \eta \norm{x} z \qquad \text{for all } x \in X^+ \setminus \{0\}.
		\end{align*}
		
		\item\label{thm:stability-for-pos-ops-interior-point:item:small-gain-interior-point-exists} 
		\emph{Interior point small-gain condition, second version:} There exists an interior point $z$ of $X^+$ and a number $\eta > 0$ such that
		\begin{align*}
			Tx \not\ge x - \eta \norm{x} z \qquad \text{for all } x \in X^+ \setminus \{0\}.
		\end{align*}
		
		\item\label{thm:stability-for-pos-ops-interior-point:item:strictly-decreasing} 
		\emph{Strong decreasing property, first version:}
		There exists an interior point $z$ of $X^+$ such that $Tz \ll z$.

		\item\label{thm:stability-for-pos-ops-interior-point:item:strictly-decreasing-lam} 
		\emph{Strong decreasing property, second version:}
		There exists an interior point $z$ of $X^+$  and $\lambda\in (0,1)$ such that 
		\begin{align}
		\label{eq:point-of-strict-decay}
			\jgc{Tz \le \lambda z}.
		\end{align}
				
		\item\label{thm:stability-for-pos-ops-interior-point:item:strong-stability} 
		\emph{Strong stability:}
		The system~\eqref{eq:Gamma-discrete-time-system} is strongly stable in the sense of Definition~\ref{def:strong-stability}.
		
		\item\label{thm:stability-for-pos-ops-interior-point:item:weak-attractivity} 
		\emph{Weak attractivity in the cone:} 
		For each $x \in X^+$ we have $\inf_{k\geq 0}\norm{T^kx} = 0$.
	\end{enumerate}
\end{theorem}

\begin{proof}
	\Implies{thm:stability-for-pos-ops-interior-point:item:stability}{thm:stability-for-pos-ops-interior-point:item:small-gain-dual} 
	If $x' \in (X')^+$ is non-zero and $T'x' \ge x'$, then we can iterate this inequality and obtain $(T')^kx' \ge x'$ for each $k \in \Z_+$. Hence, the sequence $\left((T')^kx'\right)_{k \in \Z_+}$ does not converge to $0$, so $\spr(T) = \spr(T') \ge 1$.
	
	\Implies{thm:stability-for-pos-ops-interior-point:item:small-gain-dual}{thm:stability-for-pos-ops-interior-point:item:stability} 
	Assume that $\spr(T) \ge 1$. Since the cone $X^+$ is normal and has non-empty interior, it follows that $\spr(T') = \spr(T)$ is an eigenvalue of the dual operator $T'$ with an eigenvector $x' \in (X')^+$ \cite[Corollary on p.\,315]{SchaeferWolff1999}. Hence, $T'x' = \spr(T)x' \ge x'$.
	
	\Implies{thm:stability-for-pos-ops-interior-point:item:stability}{thm:stability-for-pos-ops-interior-point:item:small-gain-interior-point-all}
	By Theorem~\ref{thm:stability-for-pos-ops} the condition $\spr(T)<1$ is equivalent to the uniform small-gain condition 
	Theorem~\ref{thm:stability-for-pos-ops}\eqref{thm:stability-for-pos-ops:itm:uniform-small-gain}, which is equivalent (for any fixed interior point $z$) due to \cite[Proposition 7.4]{MKG20} to the nonlinear version of the interior point small-gain condition, which reduces thanks to the linearity of $T$ to the condition \eqref{thm:stability-for-pos-ops-interior-point:item:small-gain-interior-point-all}, 
	using the argumentation as in Remark~\ref{rem:Nonlinear-uniform-SGC}.
	
	\Implies{thm:stability-for-pos-ops-interior-point:item:small-gain-interior-point-all}{thm:stability-for-pos-ops-interior-point:item:small-gain-interior-point-exists}
	This is clear.
	
	\Implies{thm:stability-for-pos-ops-interior-point:item:small-gain-interior-point-exists}{thm:stability-for-pos-ops-interior-point:item:stability} 
	The argument for this implication is similar as in the proof of \Implies{thm:stability-for-pos-ops-interior-point:item:stability}{thm:stability-for-pos-ops-interior-point:item:small-gain-interior-point-all}.
	
	\Implies{thm:stability-for-pos-ops-interior-point:item:stability}{thm:stability-for-pos-ops-interior-point:item:strictly-decreasing} 
	If $\spr(T)<1$, from Proposition~\ref{prop:spr-super-eigenvector} it follows that there is $\lambda \in (0,1]$ and $z\in\intt(X^+)$ such that $Tz \ll \lambda z \le z$.
	
	\Implies{thm:stability-for-pos-ops-interior-point:item:strictly-decreasing}{thm:stability-for-pos-ops-interior-point:item:strictly-decreasing-lam}
	Let $z\in \intt(X^+)$ be such that $z - Tz \in \intt(X^+)$. Then, according to 
	Lemma~\ref{prop:order-units}\eqref{prop:order-units:itm:order-unit}, there exists a number $\delta \in (0,1)$ such that $\delta z \le z - Tz$; hence, $0 \le Tz \le (1-\delta) z$. 
	
	\Implies{thm:stability-for-pos-ops-interior-point:item:strictly-decreasing-lam}{thm:stability-for-pos-ops-interior-point:item:strong-stability}
	By iterating the inequality $Tz \le \lambda z$ we obtain the inequality $0 \le T^n z \le \lambda^n z$ for each integer $n \in \Z_+$. For $x \in [0,z]$, this yields $0 \le T^n x \le \lambda^n z$ for each $n \in \Z_+$, so due to the normality of the cone we conclude that $T^n x \to 0$ as $n \to \infty$.
	
	But as $z$ is an interior point, the span of the order interval $[0,z]$ equals $X$ (Proposition~\ref{prop:order-units}\eqref{prop:order-units:itm:principal-ideal}), so $T^n x \to 0$ for each $x \in X$.
	
	\Implies{thm:stability-for-pos-ops-interior-point:item:strong-stability}{thm:stability-for-pos-ops-interior-point:item:weak-attractivity} 
	This is clear.
		
	\Implies{thm:stability-for-pos-ops-interior-point:item:weak-attractivity}{thm:stability-for-pos-ops-interior-point:item:stability}
	Since the cone $X^+$ is normal, there exists a constant $C \ge 0$ such that $\norm{x} \le C \norm{y}$ for all $x \in X$ and $y \in X^+$ that satisfy $x \in [-y,y]$ \cite[Theorem~2.38(1) and~(3)]{AliprantisTourky2007}. Moreover, it follows from Proposition~\ref{prop:order-units}\eqref{prop:order-units:itm:ball} that we can find an interior point $z$ of $X^+$ such that the order interval $[-z,z]$ contains the unit ball in $X$.
	
	According to property~\eqref{thm:stability-for-pos-ops-interior-point:item:weak-attractivity} we can find a time $k \in \Z_+$ such that $C \norm{T^k z} \le 1/2$. For this $k$, and for each $x \in X$ of norm at most $1$, we have
	\begin{align*}
		T^k x \in [-T^k z, T^k z],
	\end{align*}
	so $\norm{T^k x} \le C \norm{T^k z} \le 1/2$ for each $x$ in the unit ball of $X$. Hence, $\norm{T^k} \le 1/2$, which implies~\eqref{thm:stability-for-pos-ops-interior-point:item:stability}.
\end{proof}

It is particularly worthwhile to point out that condition~\eqref{thm:stability-for-pos-ops-interior-point:item:small-gain-dual} in Theorem~\ref{thm:stability-for-pos-ops-interior-point} is formulated in a non-uniform way (as opposed to the uniform small-gain condition in Theorem~\ref{thm:stability-for-pos-ops}\eqref{thm:stability-for-pos-ops:itm:uniform-small-gain}). As can be seen in the proof, this is because, on spaces whose cone is normal and has non-empty interior, the spectral radius of a positive operator is always an eigenvalue of the dual operator with a positive eigenvector. 
The implication \Implies{thm:stability-for-pos-ops-interior-point:item:small-gain-dual}{thm:stability-for-pos-ops-interior-point:item:stability} in the theorem is also implicitly contained in \cite[Theorem~19 and Lemma~5]{Karlin1959}.

\begin{example}
	\label{examp:Not all points are decaying} 
	Under the assumptions of Theorem~\ref{thm:stability-for-pos-ops-interior-point} the condition $\spr(T)<1$ does not imply that $Tz \ll z$ for \emph{all} interior points of $X^+$ 
	(while, according to assertion~\eqref{thm:stability-for-pos-ops-interior-point:item:strictly-decreasing} in the theorem, 
	this is true for at least \emph{one} interior point).
	For example, consider the operator
	\begin{align*}
		T:=
		\begin{pmatrix}
			\frac{1}{2} & 1           \\	
			0           & \frac{1}{2}
		\end{pmatrix}
		.
	\end{align*}
	on $\R^2$. We endow $\R^2$ with its standard cone, so $T$ is positive.
	
	Clearly, $\spr(T)=\frac{1}{2}$ and it is easy to construct,	in accordance with
	Theorem~\ref{thm:stability-for-pos-ops-interior-point}\eqref{thm:stability-for-pos-ops-interior-point:item:strictly-decreasing},
	a strictly decaying vector; for instance, we have $T(6, 2)^T = (5, 1)^T \ll (6, 2)^T$. 
	
	But at the same time the vector $(2, 2)^T$ -- which is an interior point the cone -- satisfies $T(2, 2)^T = (3,1)^T \not\ll (2,2)^T$.
\end{example}

\begin{remarks}[Computation of points of strict decay]
	\label{rem:Computation-of-points-of-strict-decay}
	\begin{enumerate}[\upshape (a)] 
		\item
		In some applications, in particular in Lyapunov-based small-gain theorems \cite{MNK21}, it is of interest to explicitly compute the number $\lambda$ and the corresponding vector $z$ as in \eqref{eq:point-of-strict-decay}, also called a \emph{point of strict decay}.
	Proposition~\ref{prop:spr-super-eigenvector} shows that one may pick any $\lambda\in (r(T),1)$, and $z := (\lambda \id - T)^{-1}y$ for any $y\in\intt(X^+)$, see \eqref{eq:Expression-point-of-strict-decay}.

		As $(\lambda \id - T)^{-1}$ is bijective from $X$ to $X$, it is an open map due to the open mapping theorem, i.e., $(\lambda \id - T)^{-1}$ maps open sets to open sets.
		This is a very nice property from the computational point of view, as even if we cannot compute $z = (\lambda \id - T)^{-1}y$ exactly for some $y\in\intt(X^+)$, 
		we nevertheless know that any sufficiently good approximation will also be a point of strict decay for the operator $T$.
		
		\item 
		If the operator $T$ is compact and \emph{strongly positive} in the sense that it maps $X^+ \setminus \{0\}$ into the interior of $X^+$, then the vector $z$ in~\eqref{eq:point-of-strict-decay} can be chosen to be an eigenvector of $T$ for the eigenvalue $\spr(T) < 1$.
		This follows from the Krein--Rutman theorem for strongly positive operators (see for instance \cite[Theorem~19.3(a)]{Deimling1985}).
	\end{enumerate}
\end{remarks}

\subsection{The quasi-compact case}
\label{subsection:positive-discrete:quasi-compact}

The characterisation of stability becomes considerably easier when the operator $T$ under consideration is compact or, more generally \emph{quasi-compact}. Here we call a bounded linear operator $T$ on a Banach space $X$ \emph{quasi-compact} if there exists an integer $n_0 \in \N$ and a compact operator $K$ on $X$ such that $\norm{T^{n_0}-K} < 1$ (alternatively, if there exists $n_1$ and a compact operator $K$ such that $\spr(T^{n_1}-K) < 1$). This is equivalent to saying that the equivalence class of $T$ in the \emph{Calkin algebra} $\calL(X) / \calK(X)$ -- where $\calK(X)$ denotes the ideal of compact operators on $X$ -- has spectral radius strictly less than $1$. The latter spectral radius is known to coincide with the so-called \emph{essential spectral radius} $\essSpr(T)$ of $T$. Hence, $T$ is quasi-compact if and only if $\essSpr(T) < 1$, and the latter condition means that on the unit circle, and outside of it, all spectral values of $T$ (if there exist any at all) are poles of the resolvent $(\argument - T)^{-1}$ with a finite-dimensional spectral space. Clearly, every compact operator and every power compact operator is quasi-compact.

Our next theorem gives additional stability criteria for positive linear operators in case that they are quasi-compact. In contrast to Theorem~\ref{thm:stability-for-pos-ops} we do not need the cone to be normal now, and moreover it suffices if the cone is total rather than generating. In case that the cone is normal and generating, though, the following criteria are complemented by those in Theorem~\ref{thm:stability-for-pos-ops}, of course. 

\begin{theorem} 
\label{thm:stability-for-pos-ops-quasi-compact}
	Let $(X,X^+)$ be an ordered Banach space with total cone and let $T \in \calL(X)$ be positive. If $T$ is quasi-compact, then the following assertions are equivalent:
	\begin{enumerate}[\upshape (i)]
		\item\label{thm:stability-for-pos-ops-quasi-compact:itm:stability} 
		\emph{Uniform exponential stability:} 
		The system~\eqref{eq:Gamma-discrete-time-system} satisfies the equivalent criteria of 
		Proposition~\ref{prop:UGAS_discrete_systems_and_Small-gain_theorem_Operators_with_Order-uniform-conditions}, 
		i.e., we have $\spr(T) < 1$.
		
		\item\label{thm:stability-for-pos-ops-quasi-compact:itm:positive-invertibility} 
		\emph{Positivity of the resolvent at $1$:}
		The operator $\id-T: X \to X$ is bijective and $(\id-T)^{-1}$ is positive. 
									
		\item\label{thm:stability-for-pos-ops-quasi-compact:itm:backwards-positive} 
		\emph{All sub-fixed vectors of $T$ are positive:}
		If $x \in X$ satisfies
		\begin{align*}
			Tx \le x,
		\end{align*}
		then $x \ge 0$.				

		\item\label{thm:stability-for-pos-ops-quasi-compact:itm:small-gain} 
		\emph{Small-gain condition:} 
		For each $0 \not= x \in X^+$ we have
		\begin{align*}
			Tx \not\ge x.
		\end{align*}
		
		\item\label{thm:stability-for-pos-ops-quasi-compact:itm:attractivity-on-cone} 
		\emph{Attractivity on the cone:} 
		For each $x\in X^+$ we have $T^kx \to 0$ as $k\to\infty$.
		
		\item\label{thm:stability-for-pos-ops-quasi-compact:itm:weak-attractivity-on-cone} 
		\emph{Weak attractivity on the cone:} 
		For each $x\in X^+$ we have $\inf_{k\geq 0}\norm{T^kx} = 0$.
	\end{enumerate}
\end{theorem}

For finite-dimensional discrete-time systems the equivalence  ``\eqref{thm:stability-for-pos-ops-quasi-compact:itm:stability} $\Leftrightarrow$ \eqref{thm:stability-for-pos-ops-quasi-compact:itm:small-gain}'' of Theorem~\ref{thm:stability-for-pos-ops-quasi-compact} and the equivalence ``\eqref{thm:stability-for-pos-ops-interior-point:item:stability} $\Leftrightarrow$ \eqref{thm:stability-for-pos-ops-interior-point:item:strictly-decreasing-lam}'' of Theorem~\ref{thm:stability-for-pos-ops-interior-point} can be found in \cite[Lemma 2.0.1]{Rue07} and \cite[Lemma 1.1]{Rue10}.
In the finite-dimensional continuous-time case, an analogue of ``\eqref{thm:stability-for-pos-ops-quasi-compact:itm:stability} $\Leftrightarrow$\eqref{thm:stability-for-pos-ops-quasi-compact:itm:small-gain}''  can be found in \cite[Theorem~1.4]{Stern1982}. 

\begin{proof}[Proof of Theorem~\ref{thm:stability-for-pos-ops-quasi-compact}]
	\Implies{thm:stability-for-pos-ops-quasi-compact:itm:stability}{thm:stability-for-pos-ops-quasi-compact:itm:positive-invertibility} 
	The argument is the same as in the proof of Theorem~\ref{thm:stability-for-pos-ops}.
	
	\Implies{thm:stability-for-pos-ops-quasi-compact:itm:positive-invertibility}{thm:stability-for-pos-ops-quasi-compact:itm:backwards-positive} 
	If $Tx \le x$, then $(\id-T)x \ge 0$ and hence $x \ge 0$ by the positivity of $(\id-T)^{-1}$.
		
	\Implies{thm:stability-for-pos-ops-quasi-compact:itm:backwards-positive}{thm:stability-for-pos-ops-quasi-compact:itm:small-gain}
	Assume that~(\ref{thm:stability-for-pos-ops-quasi-compact:itm:small-gain}) fails, i.e., there is a non-zero vector $x \in X^+$ such that $Tx \ge x$. Then $T(-x) \le -x$. But as $x$ is non-zero, the vector $-x$ is not positive; this contradicts~(\ref{thm:stability-for-pos-ops-quasi-compact:itm:backwards-positive}).
	
	\Implies{thm:stability-for-pos-ops-quasi-compact:itm:small-gain}{thm:stability-for-pos-ops-quasi-compact:itm:stability}
	Assume that $\spr(T) \ge 1$. Then we have, in particular, $\spr(T) > \essSpr(T)$. By the Krein--Rutman theorem -- or more precisely, by the version of this theorem for operators that satisfy $\spr(T) > \essSpr(T)$, see for instance \cite[Corollary~2.2]{Nussbaum1981} -- it follows that $\spr(T)$ is an eigenvalue of $T$ with an eigenvector $x \in X^+ \setminus \{0\}$. Hence,
	\begin{align*}
		Tx = \spr(T)x \ge x,
	\end{align*}
	which contradicts~(\ref{thm:stability-for-pos-ops-quasi-compact:itm:small-gain}).

	\Implies{thm:stability-for-pos-ops-quasi-compact:itm:stability}{thm:stability-for-pos-ops-quasi-compact:itm:attractivity-on-cone} This is clear.

	\Implies{thm:stability-for-pos-ops-quasi-compact:itm:attractivity-on-cone}{thm:stability-for-pos-ops-quasi-compact:itm:weak-attractivity-on-cone} This is clear.

	\Implies{thm:stability-for-pos-ops-quasi-compact:itm:weak-attractivity-on-cone}{thm:stability-for-pos-ops-quasi-compact:itm:small-gain}
	Assume that there is $x \in X^+$ such that $x\neq 0$ and $Tx \geq x$. By monotonicity of $T$ we have that 
$T^k x \geq x$ for all $k \in \Z_+$. 
	In view of weak attractivity, there is an unbounded monotone sequence $(n_k)_{k\in\N}$ such that $T^{n_k}x\to 0$ as $k\to\infty$. 
	Hence $0 = \lim_{k\to\infty}T^{n_k}x \geq x$, which implies that $x = 0$, a contradiction.
\end{proof}

Again, we note that all conditions in Theorem~\ref{thm:stability-for-pos-ops-quasi-compact} can also be replaced with analogous conditions for the dual operator; this is due to the simple fact that an operator $T$ is quasi-compact if and only if its dual operator $T'$ is quasi-compact.

\begin{remark}
\label{rem:positive-invertibility-and-quasicompact} 
	The equivalence of items \eqref{thm:stability-for-pos-ops-quasi-compact:itm:positive-invertibility} and 
\eqref{thm:stability-for-pos-ops-quasi-compact:itm:backwards-positive} in Theorem~\ref{thm:stability-for-pos-ops-quasi-compact} holds for the quasicompact operators $T$ also without the assumption of the totality of the cone. 
	The proof of the implication 
	\Implies{thm:stability-for-pos-ops-quasi-compact:itm:positive-invertibility}{thm:stability-for-pos-ops-quasi-compact:itm:backwards-positive} 
	is the same as in the proof of Theorem~\ref{thm:stability-for-pos-ops-quasi-compact}, 
	so let us show the converse implication 
	\Implies{thm:stability-for-pos-ops-quasi-compact:itm:backwards-positive}{thm:stability-for-pos-ops-quasi-compact:itm:positive-invertibility}:

	Assertion~\eqref{thm:stability-for-pos-ops-quasi-compact:itm:backwards-positive} can be rephrased by saying that the operator $\id - T$ is ``inversely positive'' in the sense that $(\id-T)x \ge 0$ implies $x \ge 0$. Let $x \in X$ be such that $(\id-T)x = 0$. By inverse positivity we have $x \geq 0$. But also $(\id-T)(-x) = 0$, and thus $-x\geq 0$; so $x = 0$, which means that $\id - T$ is injective.
		
	Since $T$ is quasi-compact, the operator $\id-T$ is a Fredholm operator of index $0$ (this follows, for instance, from \cite[Theorem~XI.5.2]{GohbergGoldbergKaashoek1990} and from the fact that the Fredholm index is constant on connected sets \cite[Theorem~XI.4.1]{GohbergGoldbergKaashoek1990}). Hence, $\id-T$ is actually bijective. Since $(\id-T)x \ge 0$ implies $x \ge 0$, the inverse $(\id-T)^{-1}$ is positive. 
	\hfill $\qed$
\end{remark}

Without the assumption $\essSpr(T) < 1$, none of the conditions~\eqref{thm:stability-for-pos-ops-quasi-compact:itm:backwards-positive} or~\eqref{thm:stability-for-pos-ops-quasi-compact:itm:small-gain} in Theorem~\ref{thm:stability-for-pos-ops-quasi-compact} is sufficient to ensure that $\spr(T) < 1$. Here are simple counterexamples:

\begin{example}
\label{examp:Multiplication operator} 
	Let $X$ denote the Banach lattice $C_b([0,+\infty))$ of bounded continuous functions with pointwise order (compare Example~\ref{ex:continuous-functions}), and let $T: X \to X$ be given for each $f \in X$ by
	\begin{align*}
		(Tf)(\omega) = (1-e^{-\omega}) f(\omega) \quad \text{for all } \omega \in [0,+\infty).
	\end{align*}
	Then $T$ is positive, has spectrum $[0,1]$ and hence spectral radius $1$; but $T$ satisfies 
	condition~\eqref{thm:stability-for-pos-ops-quasi-compact:itm:backwards-positive}, 
	as well as the property $(T(x))(s) < x(s)$ for all $x\in X^+$ and $s\geq 0$, which is a stronger property than 
	the small-gain condition \eqref{thm:stability-for-pos-ops-quasi-compact:itm:small-gain} 
	in Theorem~\ref{thm:stability-for-pos-ops-quasi-compact}.
	
	The theorem is not applicable since the essential spectral radius of $T$ is equal to $1$. 
	Also note that $Tx \not\ll x$ for all $x\in X^+$.
\end{example}

\begin{example}
	\label{ex:SS_does_not_imply_ExpSt}
	Consider the Banach lattice $X:=\ell_\infty$, ordered by its usual cone (see Example~\ref{ex:sequence-spaces}\eqref{ex:sequence-spaces:itm:ell_p}). Recall that the cone $X^+$ is normal (as it is in every Banach lattice) and has non-empty interior.
	We consider the system 
	\begin{align*}
		x(k+1)= 2R x(k),\quad k\in\Z_+,
	\end{align*}
	where $R$ is the right shift on $X$, i.e., $R$ acts on $x=(x_0,x_1,x_2,\ldots) \in \ell_\infty$ as  $Rx:=(0,x_0,x_1,x_2,\ldots)$. 
	Clearly, $R$ is a positive operator in $\calL(X)$. 
	
	Consider an arbitrary strictly positive diagonal operator $D \in \calL(X)$, defined for $x := (x_i)_{i\in\Z_+} \in X$ by $Dx:=(d_i x_i)_{i\in\Z_+}$, where $(d_i)_{i\in\Z_+}$ is a sequence of a real numbers that satisfy $0 < d_i \leq M$ for  a fixed constant $M > 0$ and all indices $i\in\Z_+$.
	
	Let $x=(x_1,x_2,\ldots) \in X^+ \setminus \{0\}$ and let $i$ be the index of the first non-zero component of $x$ (which is well-defined and finite as $x\in X^+$ and $x\neq 0$). Then the components of $2R(I+D)x$ with indices $j=0,\ldots,i$ are equal to $0$. This shows that 
	\begin{align*}
		2R(I+D)x \not\geq x \qquad \text{for all } x \in X^+ \setminus \{0\},
	\end{align*}
	which implies so-called \emph{strong small-gain condition} for the operator $2R$ used in e.g. \cite[p. 11]{DRW10}, \cite{DMS19a}, \cite{MKG20}; in particular, this implies the small-gain condition for $T:=2R$ in Theorem~\ref{thm:stability-for-pos-ops-quasi-compact}\eqref{thm:stability-for-pos-ops-quasi-compact:itm:small-gain}.
	
	The strong small-gain condition says that there are positive perturbations of the operator, under which the operator still satisfies the small-gain condition in Theorem~\ref{thm:stability-for-pos-ops-quasi-compact}\eqref{thm:stability-for-pos-ops-quasi-compact:itm:small-gain}.
	In this way it resembles the \emph{robust small-gain condition} in Theorem~\ref{thm:stability-for-pos-ops}\eqref{thm:stability-for-pos-ops:itm:perturbed-small-gain}, but note that while in the robust small-gain condition the operator is being disturbed by arbitrary additive small enough perturbations, in the strong small-gain condition above the operator is disturbed by multiplicative perturbations of a specific form. 
	
	However, $R$ is not quasicompact, and thus Theorem~\ref{thm:stability-for-pos-ops-quasi-compact} is not applicable. In fact, 
	$R$ is an \emph{isometry} as $\norm{Rx} = \norm{x}$ for each $x \in X$. Thus, also $\norm{(2R)^kx} = 2^k\norm{x} \to \infty$ as $k\to \infty$ provided that $x\neq 0$. This also shows that the robust small-gain condition is much stronger than the strong small-gain condition.
	
	Finally, note (see e.g.\ \cite[Example B.7]{HMM13}) that $\sigma(2R) = \overline{B(0,2)}$, where $B(0,2)$ is the open ball of radius $2$ with the center at $0$ in the complex plane; at the same time the point spectrum of $2R$ is empty (which already implies that the claim of Krein--Rutman theorem does not hold for $2R$).
\end{example}

\section{Stability for discrete-time linear systems: a brief tour}
\label{section:discrete-time-tour}

In this section, we give a brief overview of various known criteria for uniform exponential stability of discrete-time linear systems without any positivity assumptions. Besides comprising many important uniform stability criteria, which are currently scattered throughout the research papers, this section aims to present sufficient context for the novel results of the previous section.
We refer to \cite{Eisner2010} for many results about stability properties which are weaker than uniform exponential stability, such as {Lyapunov stability} (or {power boundedness}), {strong stability}, {weak stability}, {almost weak stability}, etc.

\subsection{Strong and weak convergence with rates}
\label{subsection:discrete:rates}

Let us briefly remark that there are also weaker stability properties of the system~\eqref{eq:Gamma-discrete-time-system} than those listed in Proposition~\ref{prop:UGAS_discrete_systems_and_Small-gain_theorem_Operators_with_Order-uniform-conditions}. One of these properties is made precise in the following definition:

\begin{definition}
	\label{def:strong-stability} 
	Let $X$ be a Banach space and let $T \in \calL(X)$.
	The system \eqref{eq:Gamma-discrete-time-system} is called \emph{strongly stable} if $T^k x \to 0$ as $k \to \infty$ for each $x \in X$.
\end{definition}

It is easy to find examples of systems \eqref{eq:Gamma-discrete-time-system} which are strongly stable but not uniformly exponentially stable. For instance, this happens if we set $X = \ell_2$ and define $T \in \calL(X)$ by
\begin{align*}
	(Tf)_n = \alpha_n f_n,\quad n \in \Z_+
\end{align*}
for each $f =(f_n)_n \in \ell_2$, where $(\alpha_n)_{n \in \Z_+}$ is a sequence in $[0,1)$ converging to $1$.

On the other hand, if the convergence of $T^kx$ to $0$ is subject to a certain rate, this already implies uniform stability, as explained in the following proposition.

\begin{proposition}
	\label{prop:discrete-time:strong-rates}
	Let $X$ be a Banach space and let $T \in \calL(X)$. 
	The following assertions are equivalent:
	\begin{enumerate}[\upshape (i)]
		\item\label{itm:discrete-time-strong-rates:stability} 
		The system~\eqref{eq:Gamma-discrete-time-system} satisfies the equivalent criteria of 
		Proposition~\ref{prop:UGAS_discrete_systems_and_Small-gain_theorem_Operators_with_Order-uniform-conditions}, 
		i.e., we have $\spr(T) < 1$.
		
		\item\label{itm:discrete-time-strong-rates:rate} 
		There exists a sequence of real numbers $0 \le a_k \to 0$ with the following property: for each $x \in X$ we can find a real number $c \ge 0$ such that
		\begin{align*}
			\norm{T^k x} \le c a_k \quad \text{for all } k \in \Z_+.
		\end{align*}
		
		\item\label{itm:discrete-time-strong-rates:integrability} 
		There exists an integrability index $p \in [1,\infty)$ such that $\sum_{k=0}^\infty \norm{T^k x}^p < \infty$ for each $x \in X$.
		
		\item\label{itm:discrete-time-strong-rates:van-neerven}
		There exists a non-decreasing function $\alpha: \R_+ \to \R_+$ 
		which is strictly positive on $(0,\infty)$ such that 
		\begin{align}
			\sum_{k=0}^\infty \alpha(\norm{T^k x}) < \infty
		\end{align}
		for each $x \in X$ of norm at most $1$.
	\end{enumerate}
\end{proposition}
\begin{proof}
	\Implies{itm:discrete-time-strong-rates:stability}{itm:discrete-time-strong-rates:rate} This is obvious; just take $a_k = \norm{T^k}$.
	
	\Implies{itm:discrete-time-strong-rates:rate}{itm:discrete-time-strong-rates:stability} We may assume that $a_k \not= 0$ for each $k$. Then the sequence $(T^k/a_k)_{k \in \Z_+}$ is bounded in operator norm due to the uniform boundedness principle, and hence the sequence $(T^k)_{k \in \Z_+}$ converges to $0$ with respect to the operator norm.
	
	\Implies{itm:discrete-time-strong-rates:stability}{itm:discrete-time-strong-rates:integrability} This follows immediately from the uniform exponential stability of $T$.
	
	\Implies{itm:discrete-time-strong-rates:integrability}{itm:discrete-time-strong-rates:van-neerven} 
	This implication is obvious.
	
	\Implies{itm:discrete-time-strong-rates:van-neerven}{itm:discrete-time-strong-rates:stability} 
	This result can be found -- under even weaker assumptions -- in \cite[Theorem~0.1(i)]{vanNeerven1995}.
\end{proof}

The equivalence between items \eqref{itm:discrete-time-strong-rates:stability} and \eqref{itm:discrete-time-strong-rates:integrability} in Proposition~\ref{prop:discrete-time:strong-rates} is known as Datko-Pazy lemma.

The conditions for uniform stability in Proposition~\ref{prop:discrete-time:strong-rates} can actually be further weakened: it suffices if we replace either the strong convergence that takes place with respect to a rate in assertion~\eqref{itm:discrete-time-strong-rates:rate}, the $p$-integrability of the orbits in assertion~\eqref{itm:discrete-time-strong-rates:integrability}, or the summability condition in assertion~\eqref{itm:discrete-time-strong-rates:van-neerven}, with corresponding \emph{weak} properties -- where \emph{weak} means that we test against functionals. Results of this type can, for instance, be found in \cite{vanNeerven1995, Mueller2001, Glueck2015}. Thus, one also obtains a characterisation of $\spr(T) < 1$ in terms of certain \emph{weak stability} properties of the operator $T$.

We finally note that, if $(X,X^+)$ is an ordered Banach space with a generating cone, then assertions~(\ref{itm:discrete-time-strong-rates:rate}) and~(\ref{itm:discrete-time-strong-rates:integrability}) in Proposition~\ref{prop:discrete-time:strong-rates} are clearly equivalent to the same conditions for vectors $x \in X^+$ only.

\subsection{Lyapunov functions}
\label{subsection:Lyapunov:functions}

\emph{Lyapunov functions} are an important tool for stability analysis. For their discussion, it is convenient to introduce the following classes of comparison functions.
We say that $\alpha:\R_+\to\R_+$ belongs to $\K$ if $\alpha(0)=0$, $\alpha$ is strictly increasing and continuous.
We say that $\alpha \in \Kinf$ if $\alpha\in\K$ and $\alpha(r)\to\infty$ as $r\to\infty$ (i.e., $\alpha$ is a homeomorphism on $\R_+$).

\begin{definition}
	\label{def:Lyapunov-function} 
	Let $X$ be a Banach space and let $T \in \calL(X)$. 
	A function $V:X \to \R_+$ is called a \emph{Lyapunov function} 
	for the system \eqref{eq:Gamma-discrete-time-system} if 
	there are $\psi_1, \psi_2 \in \Kinf$, and $\alpha\in\K$ such that the inequalities
	\begin{align*}
		\psi_1(\norm{x}) \leq V(x) \leq \psi_2(\norm{x}) \qquad \text{and} \qquad V(Tx)-V(x) \leq -\alpha(\norm{x})
	\end{align*}
	hold for each $x \in X$.
\end{definition}

Lyapunov functions can be used to characterize uniform stability. 
For linear discrete-time systems, even a weaker version of a Lyapunov function suffices 
for this purposes, as the following proposition shows:

\begin{proposition}
	\label{prop:UGAS_discrete_systems_Lyapunov-type}
	Let $X$ be a Banach space and $T\in \calL(X)$. 
	The following assertions are equivalent:
	\begin{enumerate}[\upshape (i)]
		\item\label{itm:Ulin-SGC-criteria-stability}
		The system~\eqref{eq:Gamma-discrete-time-system} satisfies the equivalent criteria of 
		Proposition~\ref{prop:UGAS_discrete_systems_and_Small-gain_theorem_Operators_with_Order-uniform-conditions}, 
		i.e., we have $\spr(T) < 1$.

		\item\label{itm:Ulin-SGC-criteria-4}
		There is an equivalent norm $\norm{\argument}_{\operatorname{equ}}$ on $X$ such that $T$ is a \emph{strict contraction} with respect to $\norm{\argument}_{\operatorname{equ}}$, i.e., there is a number  $a\in [0,1)$ such that 
		\begin{align*}
			\norm{Tx}_{\operatorname{equ}}  \leq  a \norm{x}_{\operatorname{equ}}
		\end{align*}
		for all $x \in X$.

		\item\label{itm:Ulin-SGC-criteria-5} 
		There is a Lyapunov function for the system~\eqref{eq:Gamma-discrete-time-system}.
		
		\item\label{itm:Ulin-very-weak-Lyapunov-function} 
		There is a function $V: X \to \R_+$ and a function $\alpha\in\K$ such that
		\begin{align*}
			V(Tx) - V(x) \le -\alpha(\norm{x})
		\end{align*}
		for each $x \in X$.
	\end{enumerate}
\end{proposition}

Note that, in case that condition~\eqref{itm:Ulin-SGC-criteria-4} in the above proposition is satisfied, then the norm $\norm{\argument}_{\operatorname{equ}}$ is actually a Lyapunov function for the system~\eqref{eq:Gamma-discrete-time-system}.

\begin{proof}[Proof of Proposition~\ref{prop:UGAS_discrete_systems_Lyapunov-type}]
	\Implies{itm:Ulin-SGC-criteria-stability}{itm:Ulin-SGC-criteria-4} 
	This is a well-known argument: choose a real number $s > 1$ such that the operator $sT$ still satisfies the equivalent conditions of Proposition~\ref{prop:UGAS_discrete_systems_and_Small-gain_theorem_Operators_with_Order-uniform-conditions}. Then we can define the desired equivalent norm on $X$ by 
	\begin{align*}
		\norm{x}_{\operatorname{equ}} = \sup_{k \in \Z_+} \norm{(sT)^k x}
	\end{align*}
	for each $x \in X$. Indeed, the operator $sT$ has operator norm at most $1$ with respect to this new norm, so $T$ is a strict contraction.
		
	\Implies{itm:Ulin-SGC-criteria-4}{itm:Ulin-SGC-criteria-5} 
	This is clear since the equivalent norm $\norm{\argument}_{\operatorname{equ}}$ in~\eqref{itm:Ulin-SGC-criteria-4} is itself a Lyapunov function.
	
	\Implies{itm:Ulin-SGC-criteria-5}{itm:Ulin-very-weak-Lyapunov-function} 
	This is obvious.
	
	\Implies{itm:Ulin-very-weak-Lyapunov-function}{itm:Ulin-SGC-criteria-stability} 
	For each $n \in \N$ and each $x \in X$ we obtain from~\eqref{itm:Ulin-very-weak-Lyapunov-function}
	\begin{align*}
		V(T^n x) - V(x) 
		= 
		\sum_{k=1}^n \big( V(T^k x) - V(T^{k-1} x) \big)
		\le 
		\sum_{k=1}^n -\alpha(\norm{T^{k-1}x})
	\end{align*}
	and hence,
	\begin{align*}
		\sum_{k=1}^n \alpha(\norm{T^{k-1}x}) \le V(x) - V(T^n x) \le V(x).
	\end{align*}
	Thus, we have $\sum_{k=0}^\infty \alpha(\norm{T^k x}) < \infty$ for each $x \in X$.
	By Proposition~\ref{prop:discrete-time:strong-rates}, this implies that $\spr(T) < 1$.
\end{proof}

The following remark sheds some additional light on the proof of the implication \Implies{itm:Ulin-very-weak-Lyapunov-function}{itm:Ulin-SGC-criteria-stability} in the previous proposition.

\begin{remark}
	If $\alpha\in \K$ is such that $\sum_{k=0}^\infty \alpha(\norm{T^k x}) < \infty$ holds for all $x \in X$ (as in the above proof) then, conversely, a function $V:X\to\R_+$ as in Proposition~\ref{prop:UGAS_discrete_systems_Lyapunov-type}\eqref{itm:Ulin-very-weak-Lyapunov-function} can in fact be constructed by means of the formula
	\[
		V(x):=\sum_{k=0}^\infty \alpha(\norm{T^k x}).
	\] 
	For this function $V$ one even has $V(Tx) - V(x) = -\alpha(\norm{x})$ for all $x\in X$.
\end{remark}

We note that condition~\eqref{itm:Ulin-SGC-criteria-5} implies 
Proposition~\ref{prop:UGAS_discrete_systems_and_Small-gain_theorem_Operators_with_Order-uniform-conditions}\eqref{itm:Ulin-SGC-criteria-2} even for nonlinear systems; 
it can be extended also to nonlinear systems with inputs within the input-to-state stability approach, see, e.g., \cite{JiW01} for the finite-dimensional argument, which is however absolutely analogous for infinite-dimensional systems.

If $X$ is a Hilbert space, Lyapunov functions can be constructed by solving the Lyapunov equation:

\begin{proposition}
	\label{prop:Lyapunov-equation} 
	Let $X$ be a Hilbert space. Then $r(T)<1$ if and only if there exists a positive semi-definite operator $Q \in \calL(X)$ satisfying the \emph{Lyapunov equation}
	\begin{eqnarray}
		\label{eq:Lyapunov-equation}
		T^*QT - Q = -I.
	\end{eqnarray}
	In this case, $V:x \mapsto \scalp{Qx}{x}$ is a quadratic Lyapunov function for \eqref{eq:Gamma-discrete-time-system}.
\end{proposition}

\begin{proof}
	The proof of the equivalence can be found for instance in \cite[Theorem~6.1]{Eisner2010}.

	Furthermore, due to the Lyapunov equation~\eqref{eq:Lyapunov-equation} we have
	\begin{eqnarray*}
		\|Q\|\|x\|^2 \geq \scalp{Qx}{x} = \scalp{(T^*QT + I)x}{x} = \|x\|^2 + \scalp{QTx}{Tx} \geq \|x\|^2
	\end{eqnarray*}
	for each $x \in X$. The dissipation inequality is easy to check, too: we have
	\[
		V(Tx) = \scalp{QTx}{Tx}= \scalp{T^*QTx}{x} = \scalp{(Q-I)x}{x} = V(x) - \|x\|^2.
	\]
	Hence, $V$ is indeed a Lyapunov function.
\end{proof}

Since the main purpose of our paper is to study stability of positive systems, we find it worthwhile to point out that, if the underlying space is a Banach lattice and the operator $T$ is positive, then the new norm in the Lyapunov-type condition in Proposition~\ref{prop:UGAS_discrete_systems_Lyapunov-type}(\ref{itm:Ulin-SGC-criteria-4}) can be chosen such that the Banach lattice structure is preserved:

\begin{proposition}
	\label{prop:LFs-LinSys-in-Banach-lattice}
	Let $(X,X^+)$ be a Banach lattice and $T\in \calL(X)$ be a positive operator. 
	The following assertions are equivalent:
	\begin{enumerate}[\upshape (i)]
		\item\label{itm:Ulin-SGC-criteria-BanLat-stability}
		The system~\eqref{eq:Gamma-discrete-time-system} satisfies the equivalent criteria of 
		Proposition~\ref{prop:UGAS_discrete_systems_and_Small-gain_theorem_Operators_with_Order-uniform-conditions}, 
		i.e., we have $\spr(T) < 1$.
	
		\item\label{itm:Ulin-SGC-criteria-BanLat-4} 
		There is an equivalent norm $\norm{\argument}_{\operatorname{equ}}$ on $X$ with respect to which
		$(X,X^+)$ is again a Banach lattice and with respect to which $T$ is a strict contraction.
	\end{enumerate}
\end{proposition}

\begin{proof}
	\Implies{itm:Ulin-SGC-criteria-BanLat-4}{itm:Ulin-SGC-criteria-BanLat-stability}. Again, this is obvious.
	
	\Implies{itm:Ulin-SGC-criteria-BanLat-stability}{itm:Ulin-SGC-criteria-BanLat-4}. The argument is very similar to the proof of the implication from \eqref{itm:Ulin-SGC-criteria-stability} to \eqref{itm:Ulin-SGC-criteria-4} in Proposition~\ref{prop:UGAS_discrete_systems_Lyapunov-type}: We choose $s > 1$ in the same way, but we now define 
	\begin{align*}
		\norm{x}_{\operatorname{equ}} = \sup_{k \in \Z_+} \norm{(sT)^k \modulus{x}}
	\end{align*}
	for each $x \in X$ (note the modulus around $x$). It follows from the positivity of $T$ that $\modulus{Ty} \le T\modulus{y}$ for each $y \in X$; this implies that $T$ is a strict contraction with respect to this new norm. The positivity of $T$ also implies that $(X,X^+)$ is still a Banach lattice with respect to $\norm{\argument}_{\operatorname{equ}}$.
\end{proof}

\subsection{Systems with inputs}
\label{subsection:discrete:inhomogenious}

In this subsection we add an additive input to the discrete-time system~\eqref{eq:Gamma-discrete-time-system}. So for a Banach space $X$, an operator $T \in \calL(X)$ and a mapping $u: \Z_+ \to X$, we consider the system

\begin{align}
	\label{eq:discrete-time-system-with-input}
	x(k+1) = T x(k) + u(k) \quad \text{for all } k \in \Z_+.
\end{align}

We are now interested in the question how the exponential stability of undisturbed system 	\eqref{eq:Gamma-discrete-time-system} is related to the response of the system \eqref{eq:discrete-time-system-with-input} on the inputs from certain function classes.

In the subsequent theorem we use the following notation: for a Banach space $X$ and $p \in [1,\infty]$ we denote by $\ell_p(\Z_+;X)$ the space of sequences $x = (x_n)_{n \in \Z_+}$ in $X$ for which the norm
\begin{align*}
	\norm{x}_p := 
	\begin{cases}
		\left(\sum_{k=0}^\infty \norm{x_n}_X^p \right)^{1/p} \quad & \text{if } p < \infty, \\
		\sup_{k \in \Z_+} \norm{x_n}_X \quad & \text{if } p = \infty,
	\end{cases}
\end{align*}
is finite. Similarly, we use the symbol $c_0(\Z_+; X)$ to denote the space of all sequences in $X$ that converge to $0$; this space is a closed subspace of $\ell_\infty(\Z_+;X)$ and will thus also be endowed with the norm $\norm{\argument}_\infty$.

\begin{theorem}
	\label{thm:discrite-time-convolution}
	Let $X$ be a Banach space and let $T \in \calL(X)$. 
	Fix $p \in [1,\infty)$. The following assertions are equivalent:
	\begin{enumerate}[\upshape (i)]
		\item\label{itm:discrite-time-convolution-stability} 
		\emph{Uniform exponential stability for the system without input:} 
		The system~\eqref{eq:Gamma-discrete-time-system} satisfies the equivalent criteria of 
		Proposition~\ref{prop:UGAS_discrete_systems_and_Small-gain_theorem_Operators_with_Order-uniform-conditions}, 
		i.e., we have $\spr(T) < 1$.
			
		\item\label{itm:discrite-time-convolution-l_p-with-initial-value}
		\emph{Integrable input $\Rightarrow$ integrable state:} 
		For each initial value $x(0) \in X$ and each input $u \in \ell_p(\Z_+; X)$ the solution $x$ to the system~\eqref{eq:discrete-time-system-with-input} is also in $\ell_p(\Z_+; X)$.
		
		\item\label{itm:discrite-time-convolution-l_p-without-initial-value}
		\emph{Integrable input $\Rightarrow$ integrable state for initial value $0$:} 
		For each input $u \in \ell_p(\Z_+; X)$ the solution $x$ to the system~\eqref{eq:discrete-time-system-with-input} with initial value $0$ is also in $\ell_p(\Z_+; X)$.
		
		\item\label{itm:discrite-time-convolution-l_infty-with-initial-value}
		\emph{Bounded input $\Rightarrow$ bounded state:} 
		For each initial value $x(0) \in X$ and each input $u \in \ell_\infty(\Z_+; X)$ the solution $x$ to the system~\eqref{eq:discrete-time-system-with-input} is also in $\ell_\infty(\Z_+; X)$.
		
		\item\label{itm:discrite-time-convolution-l_infty-without-initial-value}
		\emph{Bounded input $\Rightarrow$ bounded state for initial value $0$:} 
		For each input $u \in \ell_\infty(\Z_+; X)$ the solution $x$ to the system~\eqref{eq:discrete-time-system-with-input} with initial value $0$ is also in $\ell_\infty(\Z_+; X)$.
		
		\item\label{itm:discrite-time-convolution-c_0-with-initial-value}
		\emph{Convergent input $\Rightarrow$ convergent state:} 
		For each initial value $x(0) \in X$ and each input $u \in c_0(\Z_+; X)$ the solution $x$ to the system~\eqref{eq:discrete-time-system-with-input} is also in $c_0(\Z_+; X)$.
		
		\item\label{itm:discrite-time-convolution-c_0-without-initial-value}
		\emph{Convergent input $\Rightarrow$ convergent state for initial value $0$:} 
		For each input $u \in c_0(\Z_+; X)$ the solution $x$ to the system~\eqref{eq:discrete-time-system-with-input} with initial value $0$ is also in $c_0(\Z_+; X)$.
		
		\item\label{itm:discrite-time-AG}
		\emph{Asymptotic gain property:} 
		There is $C>0$ such that for each initial value $x(0) \in X$, each input $u \in \ell_\infty(\Z_+; X)$ and each number $\varepsilon > 0$ there exists a time $T>0$ such that the solution $x$ to the system~\eqref{eq:discrete-time-system-with-input} satisfies
		\begin{align*}
			\norm{x(k)} \leq \varepsilon + C \norm{u}_\infty \quad \text{for all } k\geq T.
		\end{align*}
		
		\item\label{itm:discrite-time-ISS}
		\emph{Exponential input-to-state stability:} 
		There are numbers $M>0$, $a\in (0,1)$ and $C>0$ with the following property: for each initial value $x(0) \in X$, each input $u \in \ell_\infty(\Z_+; X)$ and each $k\in\Z_+$ the solution $x$ to the system~\eqref{eq:discrete-time-system-with-input} satisfies
		\begin{align*}
			\norm{x(k)} \leq Ma^k\norm{x(0)} + C\norm{u}_\infty.
		\end{align*}
	\end{enumerate}
\end{theorem}

Theorem~\ref{thm:discrite-time-convolution}
 shows in particular that the question whether assertions~\eqref{itm:discrite-time-convolution-l_p-with-initial-value} and~\eqref{itm:discrite-time-convolution-l_p-without-initial-value} hold does not depend on the choice of $p$. 

Some of the implications in Theorem~\ref{thm:discrite-time-convolution} are known, while others seem to be folklore knowledge at best. For instance, the equivalence of~\eqref{itm:discrite-time-convolution-stability}, \eqref{itm:discrite-time-convolution-l_p-without-initial-value}, \eqref{itm:discrite-time-convolution-l_infty-without-initial-value} and~\eqref{itm:discrite-time-convolution-c_0-without-initial-value} can be found for continuous-time systems (i.e., for $C_0$-semigroups) in \cite[Theorem~5.1.2]{ArendtBattyHieberNeubrander2011}, but we do not know any reference where this is stated (and proved) for the discrete-time case. Thus, both for the sake of completeness and for the convenience of the reader, we give a complete proof of Theorem~\ref{thm:discrite-time-convolution}. Several of the subsequent arguments are close to those in the proof of \cite[Theorem~5.1.2]{ArendtBattyHieberNeubrander2011}. The equivalence of a \q{nonlinear monotone version} of the asymptotic gain property and the input-to-state stability was shown in \cite[Theorem 1]{RuS14} for finite-dimensional nonlinear monotone discrete-time systems with a continuous right-hand side.

\begin{proof}[Proof of Theorem~\ref{thm:discrite-time-convolution}]
	Throughout the proof we use the abbreviation $\calT := (T^k)_{k \in \Z_+}$, and for each sequence $u: \Z_+ \to X$ we define the convolution $\calT * u: \Z_+ \to X$ by the formula
	\begin{align*}
		(\calT * u)(k) = \sum_{j=0}^k T^{k-j}u(j) = \sum_{j=0}^k T^j u(k-j) \qquad \text{for all } k \in \Z_+.
	\end{align*}
	Then the solution $x$ to the system~\eqref{eq:discrete-time-system-with-input} with input $u$ and initial value $x(0)$ is given by the formula
	\begin{align}
		\label{eq:solution-formula}
		x(k+1) = T^{k+1} x(0) + (\calT * u)(k) \qquad \text{for all } k \in \Z_+.
	\end{align}
	Now we can prove the claimed equivalences: we are going to show 
	\begin{itemize}
		\item 
		first ``\eqref{itm:discrite-time-convolution-stability} 
		$\Rightarrow$ \eqref{itm:discrite-time-convolution-l_p-with-initial-value} 
		$\Rightarrow$ \eqref{itm:discrite-time-convolution-l_p-without-initial-value} 
		$\Rightarrow$ \eqref{itm:discrite-time-convolution-stability}'',
		
		\item 
		then ``\eqref{itm:discrite-time-convolution-stability} 
		$\Rightarrow$ \eqref{itm:discrite-time-convolution-l_infty-with-initial-value} 
		$\Rightarrow$ \eqref{itm:discrite-time-convolution-l_infty-without-initial-value}'',
		
		\item 
		then ``\eqref{itm:discrite-time-convolution-stability} 
		$\Rightarrow$ \eqref{itm:discrite-time-convolution-c_0-with-initial-value} 
		$\Rightarrow$ \eqref{itm:discrite-time-convolution-c_0-without-initial-value}'',
		
		\item 
		then ``\eqref{itm:discrite-time-convolution-l_infty-without-initial-value} 
		or \eqref{itm:discrite-time-convolution-c_0-without-initial-value} 
		$\Rightarrow$ \eqref{itm:discrite-time-convolution-stability}'', 
		
		\item 
		and finally ``\eqref{itm:discrite-time-convolution-stability} 
		$\Rightarrow$ \eqref{itm:discrite-time-ISS} 
		$\Rightarrow$ \eqref{itm:discrite-time-AG} 
		$\Rightarrow$ \eqref{itm:discrite-time-convolution-c_0-without-initial-value}''.
	\end{itemize}
	
	\Implies{itm:discrite-time-convolution-stability}{itm:discrite-time-convolution-l_p-with-initial-value}
	Let $u \in \ell_p(\Z_+;X)$. Since $\norm{T^k}$ decays exponentially, the first summand in the solution formula~\eqref{eq:solution-formula} is in $\ell_p(\Z_+;X)$. 

	As $\left(\norm{T^k}\right)_{k \in \Z_+}  \in \ell_1$, it follows from Young's inequality for convolutions of scalar-valued sequences (see \cite[Proposition 1.3.2]{ArendtBattyHieberNeubrander2011} for an analogous result for measurable functions) that
	\[
		\norm{\big(\norm{T^k}\big)_{k \in \Z_+} * \big(\norm{u(k)}\big)_{k \in \Z_+}}_{\ell_p}
		\leq
		\norm{\big(\norm{T^k}\big)_{k \in \Z_+}}_{\ell_1}
		\norm{u}_{\ell_p(\Z_+;X)},
	\]	
	which implies that $\calT * u \in \ell_p(\Z_+; X)$ .

	\Implies{itm:discrite-time-convolution-l_p-with-initial-value}{itm:discrite-time-convolution-l_p-without-initial-value} 
	This implication is obvious.
	
	\Implies{itm:discrite-time-convolution-l_p-without-initial-value}{itm:discrite-time-convolution-stability}
	Fix $y \in X$. We are going to show that the orbit $(T^k y)_{k \in \Z_+}$ is in $\ell_p(\Z_+; X)$, 
	which implies assertion~\eqref{itm:discrite-time-convolution-stability} according to Proposition~\ref{prop:discrete-time:strong-rates}\eqref{itm:discrete-time-strong-rates:integrability}.
	
	To this end, set $u := (y,0,0,\dots) \in \ell_p(\Z_+; X)$. Assertion~\eqref{itm:discrite-time-convolution-l_p-without-initial-value} implies that $\calT * u \in \ell_p(\Z_+; X)$, too. However, for each time $k \in \Z_+$, we have
	\begin{align*}
		(\calT * u)(k) = T^k y,
	\end{align*}
	so $(T^k y)_{k \in \Z_+}$ is indeed in $\ell_p(\Z_+; X)$.
		
	\Implies{itm:discrite-time-convolution-stability}{itm:discrite-time-convolution-l_infty-with-initial-value}
	This implication follows easily from the solution formula~\eqref{eq:solution-formula} and from the exponential decay of $\norm{T^k}$.
	
	\Implies{itm:discrite-time-convolution-l_infty-with-initial-value}{itm:discrite-time-convolution-l_infty-without-initial-value}
	This implication is obvious.
		
	\Implies{itm:discrite-time-convolution-stability}{itm:discrite-time-convolution-c_0-with-initial-value}
	This implication follows from the solution formula~\eqref{eq:solution-formula}, the exponential decay of $\norm{T^k}$ and the fact that the convolution of a scalar-valued $\ell_1$-sequence with a scalar-valued $c_0$-sequence is again in $c_0$.
	
	\Implies{itm:discrite-time-convolution-c_0-with-initial-value}{itm:discrite-time-convolution-c_0-without-initial-value}
	This implication is obvious.
	
	``\eqref{itm:discrite-time-convolution-l_infty-without-initial-value} or \eqref{itm:discrite-time-convolution-c_0-without-initial-value} $\Rightarrow$ \eqref{itm:discrite-time-convolution-stability}''	
	Assume that property~\eqref{itm:discrite-time-convolution-l_infty-without-initial-value} or~\eqref{itm:discrite-time-convolution-c_0-without-initial-value} holds; we set $E := \ell_\infty(\Z_+;X)$ in the former case, and $E = c_0(\Z_+;X)$ in the latter. Then the linear mapping $S: E \to E$, $u \mapsto \calT * u$ is well-defined, and it is bounded due to the closed graph theorem. 
	
	This has two consequences: as a first consequence, we obtain the estimate $\norm{T^k} \le \norm{S}$ for each $k \in \Z_+$; this follows by applying the inequality $\norm{Su}_\infty \le \norm{S} \norm{u}_\infty$ to all sequences of the form $u = (y,0,0,\dots) \in E$ (where $y \in X$). In particular, $T$ is power-bounded.
	
	The second consequence of the boundedness of $S$ is as follows. Let us fix an integer $k \in \Z_+$. Then each mapping $v: \{0,1,\dots,k\} \to X$ can be extended by zeros to a sequence in $E$, and hence obeys the estimate
	\begin{align*}
		\norm{\sum_{j=0}^k T^j v(k-j)} \le \norm{S} \norm{v}_\infty.
	\end{align*}
	For any fixed vector $y \in X$, we apply this estimate to the mapping $v: \{0,1,\dots, k\} \ni j \mapsto T^j y \in X$, and thus obtain
	\begin{align*}
		(k+1) \norm{T^k y} \le \norm{S} \sup_{j=0,\dots, k} \norm{T^j} \norm{y}.
	\end{align*}
	Since $y \in X$ and $k \in \Z_+$ were arbitrary, this shows that
	\begin{align*}
		(k+1) \norm{T^k} \le \norm{S} \sup_{j=0,\dots, k} \norm{T^j} \qquad \text{for each } k \in \Z_+.
	\end{align*}
	But we already know that $T$ is power-bounded, so it follows from the last inequality that $\norm{T^k} \to 0$ as $k \to \infty$.
		
	\Implies{itm:discrite-time-convolution-stability}{itm:discrite-time-ISS}
	This follows readily from the exponential decay of $\norm{T^k}$ and the solution formula~\eqref{eq:solution-formula}; just set $C := \sum_{k=0}^\infty \norm{T^k}$.

	\Implies{itm:discrite-time-ISS}{itm:discrite-time-AG} This is clear.
	
	\Implies{itm:discrite-time-AG}{itm:discrite-time-convolution-c_0-without-initial-value} 
	To indicate the dependence of the solution of~\eqref{eq:discrete-time-system-with-input} on the initial value and the input, we denote the solution at time $k$ by $\phi_k^u (x_0)$, where $x_0$ is the initial value and $u$ is the input.
	
	Now, fix $u \in c_0(\Z_+; X)$. We have to show that the sequence $k \mapsto \phi_k^u(0)$ is in $c_0(\Z_+; X)$, and in order to do so, let $\varepsilon > 0$. 
	
	Since $u \in c_0(\Z_+; X)$, there exists a time $k_0 \in Z_0$ such that $C \norm{u(\argument + k_0)}_\infty \leq \varepsilon$. For all times $j \in \Z_+$ we obtain, by the cocycle property,
	\begin{align*}
		\norm{ \phi_{j+k_0}^u (0) } 
		= \norm{ \phi_j^{u(\argument + k_0)} \big(\phi_{k_0}^u(0)\big) }.
	\end{align*}
	Due to the asymptotic gain property, the term on the right hand is, for all sufficiently large $j$, dominated by
	\begin{align*}
		\varepsilon + C \norm{u(\argument + k_0)}_\infty \le 2 \varepsilon.
	\end{align*}
	Hence, the sequence $k \mapsto \phi_k^u(0)$ is indeed in $c_0(\Z_+; X)$.
\end{proof}

\begin{remark}
	\label{rem:Importance:B=I} 
	The equivalence between $\spr(T)<1$ and exponential input-to-state stability of \eqref{eq:discrete-time-system-with-input} holds also if we substitute $u$ by $Bu$, where $B$ is a bounded linear operator from a Banach space $U$ to $X$, and $u: \Z_+\to U$. 
	
	But this is not the case for the other equivalences. For instance, for $B=0$ the property \eqref{itm:discrite-time-convolution-c_0-without-initial-value} reduces to the fact that $0$ is an equilibrium point of the undisturbed system, which always holds for linear systems.
The asymptotic gain property \eqref{itm:discrite-time-AG} reduces for $B=0$ to strong stability.
\end{remark}

\subsection{Stability radii}

Having a system with inputs, one can consider the robustness of the stability with respect to structured disturbances of the dynamics of the system. A prominent role in this respect was played by the concepts of complex, real and positive stability radii \cite{HiP90}.
A simple formula in terms of transfer functions for the computation of the complex stability radius for linear infinite-dimensional discrete-time systems was shown in \cite[Corollary~4.5]{WiH94}.
If the system and the structured perturbations are positive, this formula simplifies even further \cite[Theorem~3.10]{Fis97}, and moreover, for positive systems with positive structured disturbances complex, real and positive stability radii coincide \cite[Theorem~3.6]{Fis97}.
The results in \cite{Fis97} have been extended to so-called multi-perturbations in \cite{AnS08}.
In \cite{KHP06} the characterizations of spectral value sets and stability radii are obtained for linear systems with structured perturbations. These characterizations have been applied to study so-called connective stability for large-scale systems.
As a general reference for linear uncertain systems we refer to \cite[Chapter 5]{HiP11}.

\subsection*{Acknowledgements} 

A. Mironchenko is supported by the German Research Foundation (DFG)
via the grant MI 1886/2-1.

\subsection*{Data availability statement} 

Data sharing not applicable to this article as no datasets were generated or analysed during the current study.

\bibliographystyle{abbrv}
\bibliography{literature}

\end{document}